\documentclass[12pt]{amsart}
\usepackage{ amsmath, amsthm, amsfonts, amssymb, color}
 \usepackage[colorlinks, citecolor=blue,pagebackref,hypertexnames=false]{hyperref}
 \allowdisplaybreaks
\usepackage[dvips]{graphicx}
\usepackage{pstricks,pst-node,pst-text,pst-3d,epsfig,pifont,pst-plot}
\usepackage{float}

\begin{document}
\hfuzz=6pt

\widowpenalty=10000

\newtheorem{cl}{Claim}
\newtheorem{theorem}{Theorem}[section]
\newtheorem{proposition}[theorem]{Proposition}
\newtheorem{coro}[theorem]{Corollary}
\newtheorem{lemma}[theorem]{Lemma}
\newtheorem{definition}[theorem]{Definition}
\newtheorem{assum}{Assumption}[section]
\newtheorem{example}[theorem]{Example}
\newtheorem{remark}[theorem]{Remark}
\newcommand{\ra}{\rightarrow}
\renewcommand{\theequation}
{\thesection.\arabic{equation}}
\newcommand{\ccc}{{\mathcal C}}
\newcommand{\one}{1\hspace{-4.5pt}1}
\def \SL {\sqrt {\mathcal L}}

\newcommand{\nf}{\infty}
\newcommand\HH{{\mathcal H}}
\newcommand\LL{{\mathcal L}}
\newcommand\A{{\mathcal A}}
\newcommand\Diff{\mathrm{Diff}}
\newcommand\RR{\R}
\newcommand\CC{\mathbb{C}}
\newcommand\NN{\mathbb{N}}
\def\sfOh{{}^{\SF} \Omega^{1/2}}
\def\scOh{{}^{\operatorname{sc}} \Omega^{1/2}}
\def\dist{C^{-\infty}}
\def\HSL { H^p_{L, S_h}(X) }
\renewcommand\Re{\operatorname{Re}}
\renewcommand\Im{\operatorname{Im}}
\def\ang#1{\langle #1 \rangle}
\newcommand\CIdot{\dot C^\infty}

\newcommand\pihat{\hat \pi}
\newcommand\pbf{\operatorname{pbf}}
\newcommand\Vsc{\mathcal{V}_{\SC}}
\newcommand\Vkb{\mathcal{V}_{k,b}(\MMkb)}
\newtheorem{ex}[theorem]{Example}

\newcommand\restr[2]{{
  \left.\kern-\nulldelimiterspace 
  #1 
  \vphantom{\big|} 
  \right|_{#2} 
  }}
\newcommand{\mar}[1]{{\marginpar{\sffamily{\scriptsize
        #1}}}}
\newcommand{\as}[1]{{\mar{AS:#1}}}
\newcommand{\rs}[1]{{\mar{RS:#1}}}

\newcommand{\comment}[1]{\vskip.3cm
\fbox{%
\parbox{0.93\linewidth}{\footnotesize #1}}
\vskip.3cm}

\newcommand{\Ai}{{\rm Ai}}
\newcommand{\Bi}{{\rm Bi}}

\newcommand{\mc}{\mathcal}
\newcommand{\rr}{\mathbb{R}}
\newcommand{\nn}{\mathbb{N}}
\newcommand{\cc}{\mathbb{C}}
\newcommand{\hh}{\mathbb{H}}
\newcommand{\zz}{\mathbb{Z}}
\newcommand\D{\mathcal{D}}
\newcommand{\disp}[1]{\displaystyle{#1}}
\newcommand{\und}[1]{\underline{#1}}
\newcommand{\norm}[1]{\left\Vert#1\right\Vert}
\newcommand{\la}{\lambda}
\newcommand{\eps}{\varepsilon}
\newcommand{\epsi}{\varepsilon}
\newcommand{\pl}{\partial}
\newcommand{\supp}{\mbox{\rm{supp}}\,}
\newcommand{\x}{\times}
\newcommand\mcS{\mathcal{S}}
\newcommand\mcB{\mathcal{B}}
\newcommand\C{\mathbb{C}}
\newcommand\N{\mathbb{N}}
\newcommand\R{\mathbb{R}}
\newcommand\G{\mathbf{G}}
\newcommand\T{\mathbf{T}}
\newcommand\Z{\mathbb{Z}}
\newcommand\dd{\,{\rm d}}
\newcommand{\cdo}{\, \cdot \,}
\newcommand\wrt{\,{\rm d}}


\title[Bochner-Riesz profile of anharmonic oscillator $\LL=-\frac{d^2}{dx^2}+|x|$ ]
{Bochner-Riesz profile of anharmonic oscillator   $\LL=-\frac{d^2}{dx^2}+|x|$  }
\medskip

\author{Peng Chen, \  Waldemar Hebisch \ and \ Adam Sikora}
\noindent
\address{Peng Chen, Department of Mathematics, Macquarie University, NSW 2109, Australia}
\noindent
\email{achenpeng1981@163.com}
\address{
Waldemar Hebisch,
Mathematical Institute, University of Wroc{\l}aw,
Pl. Grunwaldzki~2/4,
 50-384 Wroc{\l}aw, Poland  }
\noindent
\email{hebisch@math.uni.wroc.pl }
\address{
Adam Sikora, Department of Mathematics, Macquarie University, NSW 2109, Australia}
\email{adam.sikora@mq.edu.au}

\date{\today}
\thanks{This research was supported by Australian Research Council  (A.S.).}


\begin{abstract}
We investigate spectral multipliers, Bochner-Riesz means and convergence of eigenfunction expansion  corresponding to the Schr\"odinger operator with anharmonic potential
${\mathcal L}=-\frac{d^2}{dx^2}+|x|$. We show that the Bochner-Riesz profile of the operator $\LL$ completely coincides
with such profile of the harmonic oscillator ${\mathcal H}=-\frac{d^2}{dx^2}+x^2$.
 It is especially surprising because the Bochner-Riesz profile for the one-dimensional standard Laplace operator  is known to be essentially different and the case of operators $\HH$ and $\LL$ resembles more the profile of multidimensional Laplace operators. Another surprising element of the
main obtained result is the fact that the proof is not based on
restriction type estimates and instead entirely new perspective have
to be developed to obtain the critical exponent for Bochner-Riesz
means convergence.
\end{abstract}

\maketitle

  \tableofcontents

\section{Introduction}
\setcounter{equation}{0}

One of the most significant and central problems in harmonic
analysis is convergence of the Fourier transform and series. This
problem leads in a natural way to the question of convergence of
Bochner-Riesz means of Fourier integrals and series. In a systematic
manner the topic was initiated in the 1930s by Bochner. Since then
it has attracted very significant attention. Nevertheless there
still remain many fundamental problems to be resolved. Detailed
account of the main ideas and development of this area can be found
for example in \cite[Chapter 8]{Gr09}, \cite[Section IX.2]{St},
\cite[Chapter II]{Sog0}, \cite{Tao} or  \cite{LY}.

Using the language of the spectral theory the problem  of Convergence of Bochner-Riesz means of Fourier  series can be formulated for any eigenfunction expansion of any abstract self-adjoint operators. Convergence and equivalently boundedness of Bochner-Riesz
means for general differential operators or varies specific operators were studied among the other  by Christ, Karadzhov, Koch, Ricci, Seeger, Sogge, Stempak, Tataru, Thagavelu  and Zienkiewicz, see \cite{ChS, Ka, KoRi, KoTa, SS, Sog1, Sog4, StZi, T}. See also \cite{GHS}. This paper is a continuation of these affords in particular case of the operator
${\mathcal L}=-\frac{d^2}{dx^2}+|x|$.

The theory of $L^p$ spectral multipliers is essentially equivalent to Bochner-Riesz analysis
but is more flexible and precise, see discussion in Section \ref{sec7}. Therefore we  adopt this approach in this paper and we state our main result Theorem~\ref{main} below in the
language of spectral multipliers. In this context it is worth mentioning that the theory of $L^p$ spectral multipliers itself also attracts significant interest. Initially spectral theory for self-adjoint
operators was motivated by Fourier multiplier type results of Mikhlin and H\"ormdander
\cite{Hor, Mik}. These results restricted to radial Fourier multipliers can be written in
terms of spectral multipliers for standard Laplace operators and opened question of possible
generalisation to larger class of self-adjoint operators, see also discussion in  \cite{Ch3}.
In our approach we investigate  Mikhlin and H\"ormdander multipliers  together with Bochner-Riesz analysis as essentially the same research area.
The literature devoted to the spectral multipliers is much to broad to be listed here so
we refer the reader  to \cite{CoS, COSY, DOS, GHS}  for  large class of examples of papers devoted to this area of harmonic analysis. Some recent developments going in somehow different direction can be found in \cite{MaMu}. A few other interesting examples of spectral multiplier
results in various settings can be found in \cite{Ale1, An, ChS, ClSt, MM, MuSt, SS, Sog0, Sog1, Tay}.

In  \cite{T} Thangavelu showed that the  profiles of Bochner-Riesz means convergence
for the standard Laplace operator in one dimension and   one dimensional harmonic oscillator
are essentially different. This indicates that in the theory of spectral multipliers
one has to study specific examples of operators because the results can be essential different even if
considered ambient spaces have the same topological or homogenous dimension. 

\bigskip

In this paper we consider one dimensional Schr\"odinger  type operator with anharmonic potential
\begin{equation}\label{op}
 \LL=-\frac{d^2}{dx^2}+|x|
\end{equation}
which can be precisely defined using the standard
approach of quadratic forms.
It is well-known  that this type of operator is self-adjoint and
admits a spectral resolution
$$
{\mathcal L}= \int_0^{\infty}\lambda d E_{\mathcal L}(\lambda),
$$
where the $E_\LL(\lambda)$ are spectral projectors.
For any bounded
Borel function~$F\colon [0, \infty) \to \C$, we define the
operator $F(\LL)$ by the formula
\begin{equation}\label{equw}
F(\LL)=\int_0^{\infty}F(\lambda) \wrt E_\LL(\lambda).
\end{equation}
In virtue of spectral theory the operator  $F(\LL)$ is well defined
and bounded on~$L^2(\R)$.
The operators
$\LL=-\frac{d^2}{dx^2}+|x|$ and  $\HH=-\frac{d^2}{dx^2}+x^2$ are examples of Schr\"odinger
operators with potential growing to infinity when $x$ approaches $\infty$ or $-\infty$. It is well-known that
for such operators  there exist orthonormal bases of their eigenfunctions. That is there exists a system
$\{h_n\}_{n=1}^\infty$, $h_n \in L^2 (\R)$ such that $\LL h_n=\lambda_n h_n$ and for any $f\in L^2(\R)$ we have
$\|f\|_2^2=\sum_{n=1}^\infty |\langle f, h_n\rangle| ^2$. Hence
$$
f=\sum_{n=1}^\infty h_n \langle f, h_n\rangle.
$$
The convergence in the above sum is understood in sense of $L^2(\R)$. A classical problem in harmonic
analysis is whether this series is also convergent in other $L^p(\R)$ spaces and it is one of important rationale
for developing the theory of Bochner-Riesz analysis and more general spectral multipliers. Note that now
a spectral multiplier for operator $\LL=-\frac{d^2}{dx^2}+|x|$ given by formula \eqref{equw} can be
written as
$$
F(\LL) f=\sum_{n=1}^\infty F(\lambda_n)  h_n \langle f, h_n\rangle.
$$

Spectral multiplier theorems describe
sufficient conditions for function $F$ which guarantee the operator
extends to a bounded operator acting on~$L^p$ spaces  for some
range of~$p$.

One of more interesting and significant instants of spectral multipliers are Bochner-Riesz means.
To define it we set
   \begin{equation}\label{vab1}
   \sigma^{\alpha}_R(\lambda)=
      \left\{
       \begin{array}{cl}
       (1-\lambda/R)^{\alpha}  &\mbox{for}\;\; 0\le \lambda \le R \\
       0  &\mbox{for other}\;\; \lambda . \\
       \end{array}
      \right.
   \end{equation}
We then define the operator $\sigma^{\alpha}_R(\LL)$ using
(\ref{equw}). The main problem considered in Bochner-Riesz analysis
is to find exponent $\alpha_{cr}(p)$ such that the operators
$\sigma^{\alpha}_R(\LL)$ are  bounded uniformly in $R$ on $L^p$ for
all $\alpha > \alpha_{cr}(p)$. Recall that uniform boundedness and convergence of Bochner-Riesz means are equivalent.
In addition to our discussion above  we refer readers to
\cite{CaS, H2, St, Tao} and references therein for some  further detailed
background information about Bochner-Riesz analysis and spectral multipliers.
We also want to mention that in most of the cases
full description of Bochner-Riesz profile of general differential operators or even
the standard Laplace operator is an open
problem, see \cite{ChS, SS, Sog1, Sog4}.

As we mentioned before our study is devoted to Bochner-Riesz means and spectral analysis of
particular operator $\LL=-\frac{d^2}{dx^2}+|x|$. It is motivated by
results described in \cite{AW, T}, where  combination of results
obtained by Askey, Wainger and Thangavelu provide full description
(except of the endpoints)
of  convergence of Bochner-Riesz means for the harmonic oscillator
  $\HH=-\frac{d^2}{dx^2}+x^2$ and it is one of very few examples when such full picture
  was obtained.  Also in the case of the operator $\LL$ which we consider here we obtain a complete description of the critical exponent  $\alpha_{cr}(p)$ for all $1 \le p \le \infty$.

  \medskip

One of more interesting features of our results is the fact that the
range of convergence of Bochner-Riesz means for operator $\LL$
coincides completely with the same range for harmonic oscillator. To
be more precise we note that the description of convergence of
Bochner-Riesz means which follows from
   Askey, Wainger and Thangavelu's results and which is stated in \cite[Theorem 5.5]{T} can be summarised  in the following
   way.

\begin{proposition}\label{AWT}
 Consider the operator    $\HH=-\frac{d^2}{dx^2}+x^2$. Then $\sigma^{\alpha}_R(\HH)$
  is uniformly bounded on $L^p$ if the point $(1/p,\alpha)$ belongs to regions {\rm A}
   or {\rm B}, that is if  $\alpha > \max\{0,\frac{2}{3}|\frac{1}{2}-\frac{1}{p}|-\frac{1}{6}\}$, see figure \ref{figure}.
  Next if $(1/p,\alpha)$ belongs to regions {\rm C}, that is if $\alpha < \max\{0,\frac{2}{3}|\frac{1}{2}-\frac{1}{p}|-\frac{1}{6}\}$, then
  $ \sup_{R>0} \|\sigma^{\alpha}_R(\HH)\|_{p \to p} =\infty$.
\end{proposition}

Our main result is stated in Theorem \ref{main} below. As we explain  above we prefer to formulate our main result in terms of spectral theory and it is stated in the theorem below. Then to be able to compare the Bochner-Riesz
profiles of operators $\HH$ and $\LL$  we will formulate corresponding description of
Bochner-Riesz convergence for $\LL$ in Theorem~\ref{Riesz} below.

\begin{theorem}\label{main}
Suppose that   $\LL$ is an anharmonic oscillator  defined by \eqref{op} and that
 $\supp F\subset [1/2,1]$. Assume next that  $ 1\le p \le \infty $,
 $s > \max\{\frac{1}{2},\frac{2}{3}|\frac{1}{2}-\frac{1}{p}|+\frac{1}{3}\}$
 and that $F \in H^{s}$.

 Then the operators $F(t\LL)$ are uniformly bounded on space $L^p(\R)$ and
$$
\sup_{t>0} \|F(t\LL)\|_{p\to p} \le C \|F\|_{H^{s}} < \infty.
$$
\end{theorem}
The proof of Theorem \ref{main} is described in Section \ref{sec6}. However essential preparatory ingredients of the proof are described in Sections \ref{sec4} and \ref{sec5}.
Results discussed in Section \ref{sec4} are rather standard but non-trivial so we include them
for the sake of completeness. In sections \ref{sec5} and \ref{sec6} we develop essentially new techniques for handling spectral multiplier operators.
These two sections are the most significant and interesting part of the paper. We want to stress again that
surprisingly the proof is not based on restriction type estimates as it is the case in most of known
results in Bochner-Riesz analysis. 

As we mentioned above the following result which is mainly a consequence of
Theorem~\ref{main} gives a complete picture of Bochner-Riesz convergence for the operator $\LL$.

\begin{theorem}\label{Riesz}
Suppose that   $\LL$ is   defined by \eqref{op} that is   $\LL=-\frac{d^2}{dx^2}+|x|$. Then $\sigma^{\alpha}_R(\LL)$
  is uniformly bounded on $L^p$ if
  $\alpha > \max\{0,\frac{2}{3}|\frac{1}{2}-\frac{1}{p}|-\frac{1}{6}\}$, which means  the point $(1/p,\alpha)$ belongs to regions {\rm A} or
  {\rm B}.
     Moreover if  $\alpha < \max\{0,\frac{2}{3}|\frac{1}{2}-\frac{1}{p}|-\frac{1}{6}\}$, this is if $(1/p,\alpha)$ belongs to regions~{\rm C}, then
  $ \sup_R \|\sigma^{\alpha}_R(\LL)\|_{p \to p} =\infty$.
\end{theorem}

The positive part of Theorem \ref{Riesz} is a rather straightforward consequence of
Theorem \ref{main} and the implication essentially  boils down to the fact that
$\sigma^{\alpha} \in H^{s}$ if and only if $\alpha +1/2 > s$. The negative part
essentially follows from our discussion in Section \ref{sec3} and Theorem~\ref{verlp} below.
We  conclude the proof of
Theorem \ref{Riesz} at the end of   Section \ref{sec3}.

 \begin{remark}
 We want to point out that Theorem \ref{Riesz} follows from Theorem~\ref{main} but
Theorem~\ref{main} is (at least formally) essentially stronger that Theorem \ref{Riesz}, see the discussion in
Section \ref{sec7}. The question whether the operator $\LL$ can be replaced by the harmonic oscillator  $\HH$
in the statement of Theorem~\ref{main} is an open problem.
\end{remark}

\begin{remark}
The endpoint  convergence of Bochner-Riesz means for operator $\LL$ remains an  open
question
except for $p=4$ and $p=4/3$. That is we do not know if $\sigma^{\alpha}_R(\LL)$ is uniformly bounded on $L^p$ for the critical exponent $\alpha = \max\{0,\frac{2}{3}|\frac{1}{2}-\frac{1}{p}|-\frac{1}{6}\}$. However it follows from the necessary condition described in
Section \ref{sec3} that  $\|\sigma^{0}_R(\LL)\|_{4 \to 4}$ is not uniformly bounded.
\end{remark}

The following picture describes  the convergence of Bochner-Riesz means for operators
$\LL$ and $\HH$. Note that the means are convergent in both regions $A$ and $B$.
The range $A$ is common for all abstract operators in dimension 1, for which the corresponding semigroups and heat kernels satisfies Gaussian bounds, see \cite{DOS}.
However the division between the parts $B$ (convergent)  and $C$ (divergent)  possibly  depends on the operator. Indeed in case of the standard
Laplace operator on $\R$ or on one dimensional torus  Bochner-Riesz means
converge in both regions $B$ and $C$ whereas for considered operators $\LL$ and $\HH$
the means are uniformly bounded  only in $B$ and they are not convergent in  part~$C$.
To sum up the region $A$ is completely understood for all abstract operators in dimension 1
whereas the division between regions  $B$ (convergent)  and $C$ (divergent) is not known for most
operators with exception of the standard Laplace operator, harmonic oscillator $\HH$ and
now also an operator $\LL$.

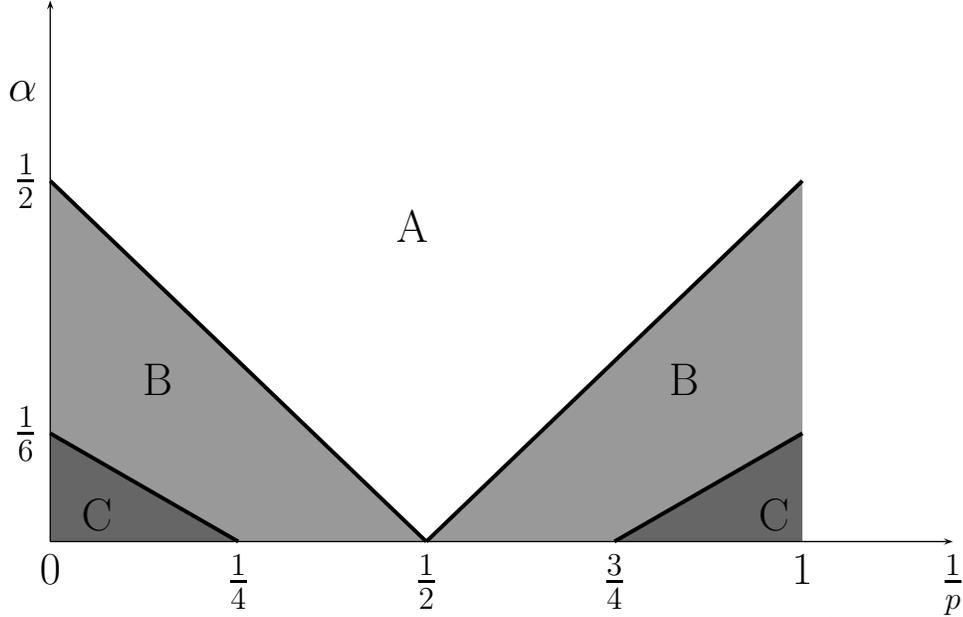
\begin{figure}[H]\label{figure}
\vspace{-6cm}
\begin{pspicture}(-1.5,14)(12,5)
    \psset{xunit=10cm, yunit=2.4cm}
     \newgray{gray00}{.9}
    \newgray{gray0}{.84}
    \newgray{gray1}{.6}
    \newgray{gray2}{.6}
    \newgray{gray3}{.4}
    \pspolygon[linestyle=none,fillstyle=solid,fillcolor=gray3](.25,0)(0,0.6)(0,0)
    \pspolygon[linestyle=none,fillstyle=solid,fillcolor=gray1](.25,0)(0,0.6)(0,2)(0.5,0)
   \pspolygon[linestyle=none,fillstyle=solid,fillcolor=gray1](.75,0)(1,0.6)(1,2)(0.5,0)
    \pspolygon[linestyle=none,fillstyle=solid,fillcolor=gray3](.75,0)(1,.6)(1,0)
    \psline[linewidth=.5pt]{->}(0,0)(1.2,0)
    \uput[d](1.2,0){\Large$\frac{1}{p}$}
    \psline[linewidth=.5pt]{->}(0,0)(0,3)
    \uput[l](0,2.5){\Large $\alpha$}
    \psline[linewidth=1.5pt](.5,0)(0,2)
    \psline[linewidth=1.5pt](.5,0)(1,2)
\psline[linewidth=1.5pt](.75,0)(1,.6)
\psline[linewidth=1.5pt](.25,0)(0,.6)
    \uput[l](0,2){\Large$\frac{1}{2}$}
    \uput[l](0.52,1.75){\Large A}
    \uput[l](0.18,0.9){\Large B}
    \uput[l](0.88,0.9){\Large B}
    \uput[l](1,0.15){\Large C}
    \uput[l](0.1,0.15){\Large C}
    \uput[l](0,0.6){\Large$\frac{1}{6}$}
    \uput[d](.5,0){\Large$\frac{1}{2}$}
      \uput[d](0,0){\Large$0$}
      \uput[d](1,0){\Large$1$}
    \uput[d](.75,0){\Large$\frac{3}{4}$}
    \uput[d](.25,0){\Large$\frac{1}{4}$}

\end{pspicture}
\vspace{6cm} \caption{Convergence of Bochner-Riesz means for operators
$-\frac{d^2}{dx^2}+|x|$ and harmonic oscillator $-\frac{d^2}{dx^2}+x^2$.
Note that for both operators the convergence in region $A$ follows from more general results which just required Gaussian upperbounds  for the corresponding semigroups, see \cite{DOS}.
  }
 \end{figure}

Through out of the paper, $W_s^p$ denotes the Soblev space
defined by the norm $\|F\|_{W_s^p}=\|(Id-\frac{d^2}{dx^2})^{s/2}F\|_{L^p}$.
Next for $p=2$ we set $W_s^p=H^s$.
$f\sim w$ means that there exist positive constants $C_1, C_2$ such
that $C_1w\leq f\leq C_2w $.

Plan of the paper. In Section \ref{sec2} we give basic description of
eigenfunction expansion of the operator $\LL=-\frac{d^2}{dx^2}+|x|$ which is based on
Airy function.
Next in Section \ref{sec4} we describe some general spectral multiplier results
which are required in the proof of the main result. In Section~\ref{sec5} we discuss in details 
properties of the Airy operator and function.
The proof of the main result that is Theorem~\ref{main} is concluded in  Section \ref{sec6}.
Next in Section \ref{sec3}  we discuss the necessary condition for convergence of Bochner-Riesz means.

\section{Eigenfunction expansion of the operator $\LL$}\label{sec2}
\setcounter{equation}{0}

We start our discussion with a precise description of the spectral
decomposition of the operator $\LL$ based on the results described in \cite{Ga}.

\medskip

We recall that the Airy function, which we denote by $\Ai$ is defined as the inverse Fourier
transform of the function $\xi \to \exp(i\xi^3/3)$, see  \cite[Definition 7.6.8, Page 213]{Ho1}.
In the sequel we will need  the following  properties of spectral decomposition of operator
$\LL$  which are proved in Section 2 of
\cite{Ga}.
\begin{proposition}\label{prop1}
Suppose that the operator $\LL$ acting on $L^2(\R)$ is defined by  formula
\eqref{op}.   Then its spectral decomposition satisfies the  following
properties:

\begin{enumerate}
\item The operator $\LL$ has only a pointwise spectrum and its eigenvalues  belong to
$(1,\infty)$. In particular  the first eigenvalue is larger than 1.

\item Every eigenvalue of $\LL$ is simple and the only point of
accumulation of the eigenvalue sequence is $\infty$.

\item The spectrum of $\LL$ is described by the following formula
$$
\{\lambda\in \R: \Ai(-\lambda)=0 \ \mbox{or} \
\Ai'(-\lambda)=0\}.
$$
Moreover, the normalized eigenfunction $\phi_n$ corresponding to
the eigenvalue $\lambda_n$ can be described as
\begin{equation}\label{e2.1}
   \phi_n(u)=
      \left\{
       \begin{array}{cl}
       A_n\Ai(u-\lambda_n)  &\mbox{for}\;\; u\geq 0 \\
       (-1)^{n+1}A_n\Ai(-u-\lambda_n)  &\mbox{for}\;\; u\leq 0 \\
       \end{array}
      \right.
   \end{equation}
where
\begin{eqnarray}\label {e2.2}
A_n=\left(2\int_{-\lambda_n}^\infty|\Ai(u)|^2du\right)^{-1/2}.
\end{eqnarray}

\item The eigenvalues $\lambda_n$  behave asymptotically in the following way
\begin{eqnarray}\label{e2.3}
\lim_{n\to \infty} \lambda_n \left(\frac{3\pi}{4}n\right)^{-2/3} =1
\end{eqnarray}
and
\begin{eqnarray}\label{e2.4}
\frac{\pi}{2}\lambda_{n+1}^{-1/2}\leq \lambda_{n+1}-\lambda_n\leq
\frac{\pi}{2}\lambda_{n}^{-1/2}
\end{eqnarray}
for all $n=1, 2, \ldots$
\end{enumerate}
\end{proposition}

\begin{proof}
Proposition~\ref{prop1} is just reformulation of Proposition 2.1,
Corollary 2.2,  Facts 2.3, 2.7 and Theorem 2.6 of \cite{Ga}. The complete asymptotic of
eigenvalues $\lambda_n$ including more precise version of
relations \eqref{e2.3} and \eqref{e2.4}  is described in \cite{FO}.
\end{proof}

\medskip

In what follows we also need the following lemma.

\begin{lemma}\label {le2.3}
Let ${\phi}_n$ is the normalized eigenfunction corresponding to the
eigenvalue $\lambda_n$  defined in \eqref{e2.1}. Then
\begin{eqnarray}\label{eq2.30}
 |{\phi}_n(u)|\leq
      \left\{
       \begin{array}{cl}
       C\lambda_n^{-\frac{1}{4}}\Big(\big||u|-\lambda_n\big|+1\Big)^{-\frac{1}{4}},  &\mbox{for}\;\; |u|\leq
       \lambda_n \\
       C\lambda_n^{-\frac{1}{4}}\exp\big(-\frac{2}{3}\big||u|-\lambda_n\big|^{\frac{3}{2}}\big), &\mbox{for}\;\; |u|>
       \lambda_n.
              \end{array}
      \right.
\end{eqnarray}
In addition
  \begin{equation}\label{eq2.31}
   \|{\phi}_n\|_{L^p}\sim
      \left\{
       \begin{array}{cl}
       \lambda_n^{\frac{1}{p}-\frac{1}{2}},  &\mbox{for}\;\; 1\leq p< 4, \\
       \lambda_n^{-\frac{1}{4}}(\ln
\lambda_n)^{\frac{1}{4}}, &\mbox{for}\;\; p=4, \\
       \lambda_n^{-\frac{1}{4}},  &\mbox{for}\;\; p>4, \\
       \end{array}
      \right.
   \end{equation}
where $f\sim w$ means that there exist positive constants $C_1, C_2$
such that $C_1w\leq f\leq C_2w $.

\end{lemma}

\begin{proof} It is well known that the Airy function $\Ai$ is bounded. In the proof, we also need the following
estimates for $\Ai$:

\noindent There exists a constant $C$ such that  for all $u>0$
\begin{eqnarray}\label{e2.5}
\big|\Ai(u)\big|\le C \exp(-2u^{\frac{3}{2}}/3)u^{-\frac{1}{4}};
\end{eqnarray}
In addition for  all $u<0$ the asymptotic
behaviour of the Airy function  as $u$ goes to minus infinity can be described in
the following way
\begin{eqnarray}\label{e2.6}
\Ai(u)=
(\pi)^{-\frac{1}{2}}|u|^{-\frac{1}{4}}\left(\sin(\frac{2}{3}|u|^{\frac{3}{2}}+\frac{\pi}{4})+O(|u|^{-\frac{3}{2}})\right),
\end{eqnarray}
 \cite[(7.6.20) and (7.6.21), Page 215]{Ho1}.

Next by (\ref{e2.2})
\begin{eqnarray}
A_n^{-2}&=&2\int_{-\lambda_n}^\infty|\Ai(u)|^2du\nonumber\\
&\sim&\int_{-\lambda_n}^{-1}|\Ai(u)|^2du+\int_{-1}^\infty|\Ai(u)|^2du\nonumber.
\end{eqnarray}
The Airy function is smooth and bounded  so by  (\ref{e2.5})
$$
\int_{-1}^\infty|\Ai(u)|^2du \le C < \infty
$$
and
$$
\int_{-\lambda_n}^{-1}|\Ai(u)|^2du\sim \int_1^{\lambda_n}
u^{-\frac{1}{2}}du \sim \lambda_n^{\frac{1}{2}}.
$$
Hence  $A_n\sim \lambda_n^{-1/4}$ so \eqref{eq2.30} follows from
\eqref{e2.1}, (\ref{e2.5}) and (\ref{e2.6}).

\medskip
Alternatively note that $\lambda_n{\phi}_n(0)^2+{\phi}'_n(0)^2=1$ so
$$
\lambda_n A_n^2|\Ai(-\lambda_n)|^2+A_n^2|\Ai'(-\lambda_n)|^2=1.
$$
Hence $A_n\sim \lambda_n^{-1/4}$  by assymptotics \eqref{e2.6} and similar assymptotics for
derivative of the Airy function described in Proposition~\ref{prop7.1} below.
\medskip

\noindent Now by  (\ref{e2.1}),
\begin{eqnarray}\label {e2.7}
\|{\phi}_n\|^p_{L^p}&=&\int_0^\infty|{\phi}_n(u)|^pdu+\int_{-\infty}^0|{\phi}_n(u)|^pdu\nonumber\\
&=&2A_n^p\int_0^\infty|\Ai(u-\lambda_n)|^pdu\nonumber\\
&=&2A_n^p\left(\int_{-\lambda_n}^{-1}|\Ai(u)|^pdu+\int_{-1}^\infty|\Ai(u)|^pdu\right).
\end{eqnarray}
The Airy function is smooth and bounded so  by  (\ref{e2.5}),
\begin{eqnarray}\label {e2.8}
\int_{-1}^\infty|\Ai(u)|^pdu\leq C<0.
\end{eqnarray}
Then by  (\ref{e2.6}), for $1\leq p<4$
$$
\int_{-\lambda_n}^{-1}|\Ai(u)|^pdu\sim
\int_{1}^{\lambda_n}u^{-\frac{p}{4}}du\sim
\lambda_n^{1-\frac{p}{4}},
$$
for $p=4$
$$
\int_{-\lambda_n}^{-1}|\Ai(u)|^pdu\sim
\int_{1}^{\lambda_n}u^{-1}du\sim \ln\lambda_n,
$$
and for $p>4$
$$
 \int_{-\lambda_n}^{-1}|\Ai(u)|^pdu\sim
\int_{1}^{\lambda_n}u^{-\frac{p}{4}}du\sim C.
$$
Now Lemma~\ref{le2.3} follows from (\ref{e2.7}), (\ref{e2.8}) and
the estimates for $A_n$.
 \end{proof}

\section {Spectral multiplier theorems for abstract self-adjoint operators}\label{sec4}
The aim of this section is to prove two auxiliary results - Lemma~\ref{2I} and
Proposition~\ref{I/4}  which we use in  the proofs of
Theorems \ref{main} and \ref{Riesz}.

Set  $I_\lambda  = [-\lambda, \lambda]$. Let $\chi_{I_\lambda}$ be the characteristic function of interval $I_\lambda$ we denote
by $I_\lambda$ also  a projection acting on $L^p(\R)$ spaces defined by
$$
I_\lambda f(x)= \chi_{I_\lambda}f(x)
$$
for any $f \in L^p(\R)$. Similarly we set  $I^c_\lambda f(x)= \chi_{I^c_\lambda}f(x)=(1-\chi_{I_\lambda})f(x)$.

We first observe that  if $\supp F \subset [1/2,1]$ then  it is enough to estimate $F(\LL/\lambda)$
on the interval $I_{2\lambda}$ and that the part of multiplier $F(\LL/\lambda)$ outside  $I_{2\lambda}$ is negligible. More precisely we show that

\begin{lemma}\label{2I}
Suppose that $\LL$ is an anharmonic oscillator defined by \eqref{op} and that $F\colon \R \to \R$ is
a continuous function such that $\supp F \subset [1/2,1]$. Then for any  $1 \le p \le \infty$
$$
\|F(\LL/\lambda)I^c_{2\lambda}\|_{p\to p } \leq
C\|F\|_{\infty}
$$
for all $\lambda >0$.
\end{lemma}
\begin{proof}
By the definition of spectral multipliers
$$
F(\LL/\lambda)I^c_{2\lambda}f = \sum_{n=1}^\infty
F(\lambda_n/\lambda){{\phi}_{n}} \langle I^c_{2\lambda}f,
{{\phi}_{n}} \rangle
$$
so
\begin{eqnarray*}
\|F(\LL/\lambda)I^c_{2\lambda}\|_{p \to p} \leq\sum_{n=1}^\infty
|F(\lambda_n/\lambda)|\|{{\phi}_{n}}\|_p
\|I^c_{2\lambda}{{\phi}_{n}}\|_{p'}.
\end{eqnarray*}

Since $\supp F \subset [1/2, 1]$ in the sum above it is enough to  consider only  such $n$ that
$\lambda_n \leq \lambda$.
It follows from  Proposition~\ref{prop1} point (1)  that $1 < \lambda_n \le \lambda$.  Hence by \eqref{eq2.31}
$$
\|{{\phi}_{n}}\|_{L^p} \leq C\lambda^{1/2}
$$
for all $1 \le p \le \infty$.
Next by \eqref{eq2.30}
$$
\|I^c_{2\lambda}{{\phi}_{n}}\|_{p'}  \leq C\exp(-\lambda)
$$
and
$$
\|F(\LL/\lambda)I^c_{2\lambda}\|_{p \to p}
\leq C\lambda^{1/2}\exp(-\lambda)\left(\sum_{\lambda_n \le \lambda} |F(\lambda_n/\lambda)|\right).
$$
By \eqref{e2.3} the number of eigenvalues below $\lambda$ is of order $\lambda^{3/2}$, so
$$
\sum_{\lambda_n \le \lambda} |F(\lambda_n/\lambda)| \leq C\lambda^{3/2}\|F\|_{\infty}.
$$
Thus
$$
\|F(\LL/\lambda)I^c_{2\lambda}\|_{p\to p } \leq
C\lambda^{2}\exp(-\lambda)\|F\|_{\infty} \le C\|F\|_{\infty}.
$$
This proves Lemma \ref{2I}.
\end{proof}

Next we shall investigate  operators $F(\LL/\lambda)I_{\lambda /4}$,
where as before $F\colon \R \to \R$ is a continuous function such
that $\supp F \subset [1/2,1]$. We obtain sufficiently precise estimates for the norm
$\|F(\LL/\lambda)I_{\lambda/4}\|_{p \to p}$. However, the range of $p$
for which the result holds is restricted to the interval  $1 \le p
\le 2$ and the estimates involve  the  norm of function $F$ in some
Sobolev spaces $H^s$ for $s>1/2$. The proof of the following proposition
follows closely an argument used in \cite[Thoerem 3.6]{CoS} and \cite[Theorem 3.2]{DOS}
which was partly motivated by results obtained by Mauceri, Meda, and  Christ in \cite{MM, Ch2}.
Some more developed versions  of this idea are described in  \cite[Theorem 4.2]{COSY} and in
\cite[Theorem 4.6]{SYY}.

\begin{proposition}\label{I/4}
Suppose that $\LL$ is an anharmonic oscillator defined by \eqref{op} and that $F\colon \R \to \R$ is
a  function such that $\supp F \subset [1/2,1]$ and $F \in H^s(\R) $ for some $s>1/2$.
Then for any  $1 \le p \le 2$
$$
\|F(\LL/\lambda)I_{\lambda/4}\|_{p\to p } \leq
C\|F\|_{H^s}
$$
for all $\lambda >0$.
\end{proposition}
\begin{remark}\label{3.3} {\rm Note that the same argument as in Section \ref{sec3} below shows that
Proposition~\ref{I/4} does not hold any longer if  $p>4$. More precisely for any $p>4$
there exists $s>1/2$ such that estimate from Proposition \ref{I/4} is not satisfied. }
\end{remark}
\begin{proof}[Proof of Remark \ref{3.3}]
Let $\eta \in C_c^\infty(\R)$ is   such a function that $\eta(0)=1$ and  $\supp
\eta \subset [-\frac{\pi}{2},\frac{\pi}{2}]$ and set
$F_n(\tau)=\eta((4\tau/3-1)\lambda_n\sqrt{\lambda_{n+1}})$.  Then for some small $\epsilon$ and $n$ large enough one has $\supp F_n \subset [3/4-\epsilon, 3/4+\epsilon] $. Now if Proposition~\ref{I/4} holds for some $p>4$ then
$$
\left\|F_n\left(\frac{\LL}{\frac{4}{3}\lambda_n}\right)I_{\frac{4}{3}\lambda_n/4} \right\|_{p\to p } \le C\|F_n \|_{H^s} \le C\lambda_n^{3s/2-3/4}
$$
However the same argument as in the proof of Theorem~\ref{verlp} shows that
\begin{eqnarray*}
\left\|F_n\left(\frac{\LL}{\frac{4}{3}\lambda_n}\right)I_{\frac{4}{3}\lambda_n/4} \right\|_{p\to p }
=
\|I_{\lambda_n/3}{\phi}_n\|_{p'}\|{\phi}_n\|_{{p}} \ge
c\lambda_n^{-1/4}\lambda_n^{1/p'-1/2}
\end{eqnarray*}
Thus $3s/2-3/4> 1/p'-3/4$ that is $\frac{2}{3p'}<s$ and Proposition~\ref{I/4} cannot hold
for any $p>4$.
\end{proof}

The rest of this section is devoted to the proof of Proposition~\ref{I/4}.
 We split its proof into a few separate statements.
 We start with  recalling a useful notation coming from \cite{CoS}.
For any function $F\colon \R \to \R$  and any parameter  ${M} \in (1,\infty)$ we set
\begin{equation}\label{M}
 \|F\|_{{M},q}=\left(\frac{1}{{M}}\sum_{l=-\infty}^{\infty}\sup_{\theta\in [\frac{l-1}{{M}},\frac{l}{{M}})}
 |F(\theta)|^q\right)^{1/q}.
 \end{equation}
 The following lemma  plays significant role in this section and  in Section \ref{sec6} below,
 see  Lemma \ref{G2-L1}.
 Its proof is straightforward modification of the argument used in  \cite[(3.29)]{CoS} and
 \cite[Proposition 4.6]{DOS}.

\begin{lemma}\label{sup-sum}
Suppose that $s>0$ and that  $\xi \in C_c^{\infty}$ is a function such that
{\rm supp}~$\xi \subset [-1,1]$,   $\widehat{\xi}(0)=1$   and
$\widehat{\xi}^{(k)}(0)=0$  for all $1\le k  \le s+2$.  Next set
$\xi_{{M}}(\theta)={M}\xi({M}\theta)$  and assume that
 $G\colon \R \to \R$. Then
$$
\|G-G*\xi_{M}\|_{{M},q} \le C{M}^{-s}\|G\|_{W^q_s}.
$$
for all $s>1/q$.
Moreover
$$
\|G*\xi_{M}\|_{{M},q} \le \|G\|_{q}
$$
and
$$
\|G\|^q_{{M},q}
\leq C\left(\|G\|^q_{q} + {M}^{-qs}\|G\|^q_{W^q_s}\right).
$$
\end{lemma}
\begin{proof}
For the proof of the first inequality we refer readers to \cite[(3.29)]{CoS} or \cite[Proposition 4.6]{DOS}.
To show the second estimate note that
\begin{equation*}
| \xi_{{M}}*G(\theta)|^q \leq   \|\xi_{M} \|^q_{L^{q'}}
  \int_{\theta-1/{M}}^{\theta + 1/{M}} |G(\theta')|^q \dd \theta',
\end{equation*}
so
\begin{eqnarray*}\label{f: Np estimate 2}
\|\xi_{M}* G\|_{{M},q}  &=&
\bigg( \frac{1}{{M}}\sum_{i=-\infty}^{\infty}
\sup_{\theta\in  [\frac{i-1}{{M}},\frac{i}{{M}})}|
\xi_{M}* G(\theta)|^q \bigg) ^{1/q}  \nonumber
\\  &\leq&   \frac{\| \xi_{M}\|_{L^{q'}}}{{M}^{1/q}}
\bigg( \sum_{i=-\infty}^{\infty}\int_{(i-2)/{M}}^{(i+1)/{M}}
|G(\theta')|^q \dd \theta' \biggr) ^{1/q}
\leq \frac{3\|\xi_{M}\|_{L^{q'}}}{{M}^{1/q}}  \|G\|_{L^q}  \leq C
\|G \|_{L^q} ,
\end{eqnarray*}
see also \cite[(4.9)]{DOS}. This proves the second estimate. The third estimate is a direct consequence of first two.
\end{proof}

The next step in the proof of Proposition~\ref{I/4} is to establish some  partial restriction type estimate result. Note that the global version (without projection $I_{\lambda/4}$) of
such restriction estimate is false. Indeed examining the proof of Proposition 4.8 below shows that
without projection $I_{\lambda/4}$ the Lemma \ref{I4res} can only hold if
the norm $\|F\|^2_{\lambda^{3/2},2}$ is replaced by the stronger norm $\|F\|^2_{\lambda^{3/2},4+\epsilon}$.
 We want to point out also that one has to apply following estimates to the operator $F*\xi_{M}$
so it is necessary to assume that we consider functions with support slightly outside the interval
$[1/2,1]$.  Note also  that  $ \|F(\LL/\lambda)\|_{p \to p} \le C_{\lambda_0} \|F\|_\infty$   for
all $\lambda \le \lambda_0$, any fixed $\lambda_0$ and all $ 1\le p \le \infty$.
In fact any $L^p \to L^q$ norm satisfies such estimates.
 Hence  it is enough to
consider large $\lambda$ that is $\lambda$ bigger than some fixed constant.

\begin{lemma}\label{I4res}
Suppose that $\LL$ is an anharmonic oscillator defined by \eqref{op} and that $\lambda >4$.
Assume also  that $F\colon \R \to \R$ is
a  function such that $\supp F \subset [3/8, 9/8 ]$.


Then for any  $1 \le p \le \infty$
\begin{eqnarray}\label {e2.9}
\|F(\LL/\lambda)I_{\lambda/4}\|^2_{1 \to 2 }= \sup_{|y| \le
{\lambda/4}}\int_\R|K_{F({{\LL}/\lambda})}(x,y)|^2dx\leq
C\lambda^{1/2}\|F\|^2_{\lambda^{3/2},2}
\end{eqnarray}
where $\|F\|^2_{\lambda^{3/2},2}$ is the norm defined by \eqref{M} with $M=\lambda^{3/2}$.
\end{lemma}
\begin{proof}
By orthonormality of the  eigenfunction expansion
\begin{eqnarray*}
 \int_\R|K_{F({{\LL}/\lambda})}(x,y)|^2dx=\sum_{k=[3\lambda^{3/2}/8]+1}^{[9\lambda^{3/2}/8]+1}
\|K_{(\chi_{\left[\frac{k-1}{\lambda^{3/2}},\frac{k}{\lambda^{3/2}}\right)}F)(\LL/\lambda)}(\cdot,y)\|_{L^2}^2.
\end{eqnarray*}
Note next that if $ \lambda_{n+1} \le 9\lambda/8$ then by \eqref{e2.4}
$$
\lambda_{n+1}-\lambda_{n} \ge \frac{\pi}{2}\lambda_{n+1}^{-1/2} \ge
\lambda^{-1/2}.
$$
 Hence there is at most one number of the form $\lambda_{n}/\lambda$ which
belongs  to interval $\left[\frac{k-1}{\lambda^{3/2}},\frac{k}{\lambda^{3/2}}\right)$.
Thus
\begin{eqnarray}\label {e3.3}
 \int_\R|K_{F({{\LL}/\lambda})}(x,y)|^2dx&=&\sum_{k=[3\lambda^{3/2}/8]+1}^{[9\lambda^{3/2}/8]+1}
\|K_{(\chi_{\left[\frac{k-1}{\lambda^{3/2}},\frac{k}{\lambda^{3/2}}\right)}F)(\LL/\lambda)}(\cdot,y)\|_{L^2}^2\nonumber\\
&=&\sum_{k=[3\lambda^{3/2}/8]+1}^{[9\lambda^{3/2}/8]+1}
\int_\R\big|\sum_{\lambda_n\in\left[\frac{k-1}{\lambda^{1/2}},\frac{k}{\lambda^{1/2}}\right)}F(\lambda_n/\lambda){\phi}_n(x){\phi}_n(y)\big|^2dx\nonumber\\
&\le &\sum_{k=[3\lambda^{3/2}/8]+1}^{[9\lambda^{3/2}/8]+1}\sup_{\theta \in
\left[\frac{k-1}{\lambda^{3/2}},\frac{k}{\lambda^{3/2}}\right)}|F(\theta)|^2|{\phi}_n(y)|^2.
\end{eqnarray}
The eigenfunction ${\phi}_n$ in the last line of the estimates above corresponds to the
unique $\lambda_n$ such that $\lambda_n \in\left[\frac{k-1}{\lambda^{1/2}},\frac{k}{\lambda^{1/2}}\right)$ and if such eigenvalue
does not exist it should be replaced by $0$.
However  if $|y| \le \lambda/4$ and $\lambda_n  \in [3\lambda/ 8,9\lambda/8]$ then
by \eqref{eq2.30}
$$
|{\phi}_n(y)|^2\le C |\lambda_n| ^{-1} \le C \lambda^{-1}.
$$
Thus \eqref{e2.9} follows from \eqref{e3.3}.

\end{proof}

The next ingredient required for our main argument is a simple lemma described in
\cite{DOS}. Recall that for any positive potential $V\in L_{loc}^1(\R^d)$ we can
define the operator $L=-\Delta_d+V$ by the standard quadratic forms approach.

\begin{lemma}\label {le3.2}
Let $L=-\Delta_d+V$, where $V\in L_{loc}^1(\R^d)$ and $V\geq 0$.
Suppose that for some $c>0$
$$
\int_{\R^d}(1+V(x))^{-c}dx<\infty.
$$
Then
\begin{eqnarray}
\|(1+L)^{-c/2}\|_{L^2\to L^1}<C\int_{\R^d}(1+V(x))^{-c}dx.\nonumber
\end{eqnarray}
\end{lemma}
 \begin{proof}
 For the proof we refer readers to \cite[Lemma 7.9]{DOS}.
 \end{proof}

Following corollary is a straightforward consequence of
Lemmas~\ref{I4res} and \ref{le3.2}
\begin{coro}\label{coro}
Suppose that $\LL$ is an anharmonic oscillator defined by \eqref{op}.
Assume also  that $F\colon \R \to \R$ is
a  function such that $\supp F \subset [1/4, 2 ]$. Then   for every  $\varepsilon
>0$ there exists a constant $C_{\varepsilon}$ such that
\begin{eqnarray*}
\|F(\LL/\lambda)I_{\lambda/4}\|^2_{1\to 1 } \leq
C_{\varepsilon}\lambda^{(3/2+\varepsilon)}\|F\|^2_{\lambda^{3/2},2}.
\end{eqnarray*}
for all $\lambda >4$.
\end{coro}
 \begin{proof}
It is enough to note that if $c=1+\varepsilon$ and $G(\theta) =(1+\lambda\theta)^{c/2}F(\theta)$  then
\begin{eqnarray*}
\|F(\LL/\lambda)I_{\lambda/4}\|^2_{1\to 1 }\le
\|(1+\LL)^{c/2}F(\LL/\lambda)I_{\lambda/4}\|^2_{1\to 2 }
\|(1+\LL)^{-c/2}\|^2_{L^2\to L^1}\\ \le C_{\varepsilon}
\lambda^{1/2}\|G\|^2_{\lambda^{3/2},2} \le C_{\varepsilon}
  \lambda^{(3/2+\varepsilon)}\|F\|^2_{\lambda^{3/2},2}.
\end{eqnarray*}
\end{proof}

The rest of the proof of Proposition~\ref{I/4} is now a straightforward modification of argument used in
\cite[Lemma 3.4]{CoS} or in \cite[Section 4]{DOS}. Therefore here we only sketch the proof to show the role and  significance of Lemma~{\rm\ref{I4res}}.

\begin{proof}[{ Proof of Proposition~\ref{I/4}}]
It is enough to show that for any  ${\varepsilon} >0$
$$
\|F(\LL/\lambda)I_{\lambda/4}\|_{1 \to 1 }\leq
C\|F\|_{H^{1/2+{\varepsilon}}}.
$$
To do this consider function $\xi_\lambda$ defined in Lemma \ref{sup-sum}
and set
$$
F(\LL/\lambda
)=(F-F*\xi_{\lambda^{3/2}})(\LL/\lambda)+F*\xi_{\lambda^{3/2}}(\LL/\lambda).
$$
Recall that   $ \|F(\LL/\lambda)\|_{p \to p} \le C \|F\|_\infty$  for any $\lambda < \lambda_0$, any fixed $\lambda_0$ and $ 1\le p \le \infty$. Hence  it is enough to
consider large $\lambda$  and we can assume that $\supp F*\xi_{\lambda^{3/2}} \subset [3/8,9/8]$. Now  by Corollary \ref{coro} and Lemma \ref{sup-sum}
\begin{eqnarray}\label{e6.4}
\|(F-F*\xi_{\lambda^{3/2}})(\LL/\lambda)I_{\lambda/4}\|^2_{1 \to 1}&\leq&
C\lambda^{3/2+\varepsilon}\|(F-F*\xi_{\lambda^{3/2}})\|_{\lambda^{3/2},2}^2\nonumber\\
&\leq&C\lambda^{3/2+\varepsilon} \lambda^{-2(3/2)(1/2+\varepsilon/3)}
\|F\|^2_{H^{1/2+\varepsilon/3}}\nonumber\\
&\leq&C\|F\|^2_{H^{1/2+\varepsilon/3}}.
\end{eqnarray}

\bigskip

To estimates the term corresponding to $F*\xi_{\lambda^{3/2}}$
we note that by Lemmas \ref{I4res} and \ref{sup-sum}
\begin{eqnarray*}
\|F*\xi_{\lambda^{3/2}}(\LL/\lambda) I_{\lambda/4}\|^2_{1 \to 2} \leq
 C\lambda^{1/2} \|F*\xi_{\lambda^{3/2}}\|_{\lambda^{3/2},2}^2
\le
C\lambda^{1/2}\|F\|^2_2.
\end{eqnarray*}
Equivalently the above inequality can be stated as
\begin{eqnarray}\label{mm}
\sup_{|y| \le  \lambda/4} \int_{\R}|K_{F*\xi_{\lambda^{3/2}}(\LL/\lambda)}(x,y) |^2dx \leq
C \lambda^{1/2}\|F\|^2_2.
\end{eqnarray}
However we recall that  for any  operator satisfying Gaussian-type heat
kernel bounds the following basic estimate holds, see \cite[(4,4) and (4.5)] {DOS} or  \cite{He95}
\begin{eqnarray*}
\sup_{|y| \le  \lambda/4} \int_{\R}|K_{F*\xi_{\lambda^{3/2}}(\LL/\lambda)}(x,y) |^2(1+\lambda^{1/2}|x-y|)^s dx
\leq C |B(y, \lambda^{1/2})|\|F*\xi_{\lambda^{3/2}}\|^2_{H^{(s+1)/2+\varepsilon}}\\
=C \lambda^{1/2}\|F*\xi_{\lambda^{3/2}}\|^2_{H^{(s+1)/2+\varepsilon}}\le C \lambda^{1/2}\|F\|^2_{H^{(s+1)/2+\varepsilon}}.
\end{eqnarray*}
Now one can  use  Mauceri-Meda interpolation trick, see \cite[Lemma 4.3]{DOS}. That is, we can consider the above
estimates with large  $s$ and interpolate  with inequality \eqref{mm} to show that
\begin{eqnarray*}
\sup_{|y| \le  \lambda/4} \int_{\R}
|K_{F*\xi_{\lambda^{3/2}}(\LL/\lambda)}(x,y) |^2(1+\lambda^{1/2}|x-y|)^{1+\varepsilon'} dx
\le C \lambda^{1/2}\|F\|^2_{H^{1/2+\varepsilon''}}.
\end{eqnarray*}
for all $\varepsilon' < \varepsilon'' $
 Alternatively one can prove that the above estimates follows from Lemma \ref{I4res} using the finite propagation speed
 for the  wave equation technique, see   \cite[(3.10) and (3.28)]{CoS}.
 The last estimate and the Cauchy-Schwarz inequality yield
 \begin{eqnarray*}
\sup_{|y| \le  \lambda/4} \int_{\R}
|K_{F*\xi_{\lambda^{3/2}}(\LL/\lambda)}(x,y) | dx
\le C  \|F\|^2_{H^{1/2+\varepsilon''}}.
\end{eqnarray*}
 Thus
 \begin{eqnarray*}
\|F*\xi_{\lambda^{3/2}}(\LL/\lambda) I_{\lambda/4}\|^2_{1 \to 1} \le C
\|F\|^2_{H^{1/2+\varepsilon}}.
\end{eqnarray*}
This finishes the proof of Proposition~\ref{I/4}.
\end{proof}

\bigskip

One of most surprising points of our approach is the fact that the optimal spectral multiplier result
cannot be obtained as a consequence of restriction type estimates and one has to develop new techniques to obtain the optimal Bochner-Riesz summability exponent. To illustrate this point we will
sketch the proof of Proposition~\ref{th3.1} below even though this result does not give
the sharp Bochner-Riesz summability  result described in Theorem~\ref{Riesz}. Inspecting the proof it is easy to
note that estimate \eqref{e20.9} fails for $q<4$ so this approach does not lead to the optimal
result for Bochenr -Riesz summability. Property \eqref{e20.9} is a discrete version of restriction type estimates close in nature to Sogge's cluster type estimates, see for example \cite[Theorem 2.2 and Corollary 2.3]{CoS}. It is interesting to note here that in the setting of the standard
Laplace operator of classical Fourier Transform the Bochner-Riesz and Restriction conjecture are
very closely related, see \cite{T10}.

 \begin{proposition}\label{th3.1}
 Suppose that operator $\LL$ is defined by formula \eqref{op} and  assume also that   $\supp F\subset [-1,1]$. Then for all $1 \le p \le \infty$ and any $s>1/2$,
$$
\sup_{t>0}\|F(t\LL)\|_{L^p\to L^p}\leq C\|F\|_{W_s^4},
$$
where $\|F\|_{W_s^4}=\|(1-d_x^2)^{s/2}F\|_{4}$ is the norm in $L^4$ Sobolev space of order $s$.
\end{proposition}

\begin{proof}
It is enough to prove Proposition~\ref{th3.1} for p=1. The rest of the range follows then by
self-adjointness and interpolation.
Using the same approach as in Proposition~\ref{I/4}, Lemma~{\rm\ref{I4res}} and
Corollary~\ref{coro}
 it is not difficult to note that to prove Proposition~\ref{th3.1} it is enough to show the following version
 of  estimate (\ref{e2.9})
 \begin{eqnarray}\label {e20.9}
\|F(\LL/\lambda)\|^2_{1 \to 2 }= \sup_{y}\int_\R|K_{F({{\LL}/\lambda})}(x,y)|^2dx\leq
C\lambda^{1/2}\|F\|^2_{\lambda^{3/2},q}
\end{eqnarray}
 for
$q=4+\varepsilon$, for all $\varepsilon>0$ and for all functions $F$
such that  $\supp F\subset [-1, 1]$.

 We are going
to prove  estimate (\ref{e20.9}) only for $y=\lambda$ as the proof for other
$y\in \R$ is similar or simpler. By \eqref{e3.3}, estimate \eqref{eq2.30}
and H\"older's inequality
\begin{eqnarray*}
 \int_\R|K_{F({{\LL}/\lambda})}(x,y)|^2dx=\sum_{k=1}^{[\lambda^{3/2}]+1}\sup_{\theta \in
\left[\frac{k-1}{\lambda^{3/2}},\frac{k}{\lambda^{3/2}}\right)}|F(\theta)|^2|{\phi}_n(y)|^2\\
\le C\sum_{k=1}^{[\lambda^{3/2}]+1}\sup_{\theta \in
\left[\frac{k-1}{\lambda^{3/2}},\frac{k}{\lambda^{3/2}}\right)}|F(\theta)|^2\lambda^{-1/2}
(|\lambda -k\lambda^{-1/2}|+1)^{-1/2} \\ \le
C\lambda^{1/2}\|F\|^2_{\lambda^{3/2},2p}
\left( \lambda^{-3/2}\sum_{k=1}^{[\lambda^{3/2}]+1}
\left(\frac{\lambda}{\lambda|1 -k\lambda^{-3/2}|+1}\right)^{p'/2}  \right)^{1/p'}.
\end{eqnarray*}
Recall that the eigenfunction ${\phi}_n$ in the above  estimates corresponds to the
unique $\lambda_n$ such that $\lambda_n \in\left[\frac{k-1}{\lambda^{1/2}},\frac{k}{\lambda^{1/2}}\right)$ and if such eigenvalue
does not exist it should be replaced by $0$. Note that the last sum is uniformly bounded independently of $\lambda$ if $p'<2$. This shows
(\ref{e20.9}) for any $q=2p>4$ and finishes the proof of Proposition \ref{th3.1}.
\end{proof}

\section{More light on Airy function }\label{sec5}
Consider  the Airy operator which formally defined by the formula
\begin{equation}\label{Airy}
{\A} = -\frac{d^2}{dx^2} + x.
\end{equation}
The Airy function $\Ai $ which we recall in Section~\ref{sec2} is a
bounded on ${\mathbb R}$ solution of of the equation  $\A f = 0$.
Another linearly independent solution of this equation function
$\Bi$  grows exponentially as $x \to \infty$ so it is not a tempered
distribution and is not relevant to our discussion here.

Using just function  $\Ai$ we can describe complete system of
eigenfunctions of ${\A}$.  Set ${\varphi}_\lambda(x) = \Ai(x -
\lambda)$. For any function $f\in L^2(\R)$ we define the Airy
Transform by the following formula
$$
{\mathcal T}{f}(\lambda) = \langle f, {\varphi}_{\lambda} \rangle =
(f*{\check \Ai})(\lambda),
$$
where ${\check \Ai}(\lambda)={ \Ai}(-\lambda)$.

Since ${\widehat \Ai}(\omega) = \exp(i\omega^3/3)$, mapping
 ${\mathcal T}$ is an isometry on $L^2(\R)$
and its inverse is given by
$$
{\mathcal T^{-1}}{g}(x) =\Ai*{ g}(x).
$$
for any $g\in  L^2(\R)$.
\begin{lemma}\label{poz}
Suppose that $F \colon \R \to \R $ is a bounded function and let
$F({\A})$ be the spectral multiplier corresponding to function $F$
and the Airy operator ${\A}$. Then
$$
{\mathcal T} (F(\A) f)(\lambda) = F(\lambda) {\mathcal T}{f}(\lambda)
$$
for all $f\in L^2(\R)$. In addition   $K_{F(\A)}(x,y)=F(\A)\delta_y(x)$ - the kernel of the operator $F(\A)$ is described by the formula
\begin{eqnarray*}
K_{F(\A)}(x,y)=\int_{-\infty}^{\infty} F(\lambda) {\varphi}_\lambda(x){\varphi}_\lambda(y) d\lambda =\int_{-\infty}^{\infty} F(\lambda) \Ai(x - \lambda)\Ai(y - \lambda) d\lambda.
\end{eqnarray*}
Moreover
 \begin{equation*}
\int_\R |K_{F(\A)}(x,y)|^2dy=\int_\R|F(\lambda) \Ai(x - \lambda)|^2d\lambda
\end{equation*}
for all $x\in \R$.
\end{lemma}
\begin{proof}
Lemma \ref{poz} follows from the definition of the Airy transform ${\mathcal T}$
and the following simple observation
$$
{\A}{\varphi}_\lambda = \lambda{\varphi}_\lambda
$$
by a standard argument.
\end{proof}

In the sequel it will be convenient to use the following
 description of the asymptotic
of  the Airy function which for $x$ negative  is a slightly more precise version of
estimate \eqref{e2.6}
\begin{lemma}\label{prop7.1}
The Airy function can be expanded as
\begin{equation}\label{ai-asymp}
\Ai(x) = \exp(i{\zeta}(x))\theta(x) + \exp(-i{\zeta}(x)){\overline
\theta}(x)
\end{equation}
where ${\zeta}(x) = 2|x|^{3/2}/3$ for $x < -1$. Moreover,  function
${\zeta}$ and all of its  derivatives are bounded for $x \geq -1$
and $|d^k_x \theta|(x) \leq C_k(1 + |x|)^{-k - 1/4}$ for all $x\in
\mathbb{R}$.
\end{lemma}
\begin{proof}
Function \Ai(x) is an entire analytic function of $x\in \mathbb{C}$ and
by \cite[(7.6.18), Page  214]{Ho1}
 for all $x\in \mathbb{C}$
 \begin{eqnarray*}
\Ai(x)=-\omega_1\Ai(\omega_1x)-\omega_2\Ai(\omega_2x)
\end{eqnarray*}
where $\omega_1=e^{i\pi/3}=\frac{-1}{2}+\frac{\sqrt{3}}{2}i$ and
$\omega_2=\omega_1^2 =\frac{-1}{2}-\frac{\sqrt{3}}{2}i$. For $x<-1$, by
 \cite[(7.6.19), Page  214]{Ho1},
\begin{eqnarray*}
\Ai(\omega_1x)=\exp(-2i|x|^{3/2}/3)(2\pi)^{-1}\int_{-\infty}^\infty
\exp({-\xi^2\sqrt{|x|}\omega_3+i\xi^3/3}) d\xi
\end{eqnarray*}
and
\begin{eqnarray*}
\Ai(\omega_2x)=\exp(2i|x|^{3/2}/3)(2\pi)^{-1}\int_{-\infty}^\infty
\exp(-\xi^2\sqrt{|x|}\omega_4+i\xi^3/3)d\xi
\end{eqnarray*}
where $\omega_3=-\overline{\omega_1}$ and
$\omega_4=\overline{\omega_2}$. Thus for $x<-1$,
\eqref{ai-asymp} holds with
$$\theta(x)=-\omega_2(2\pi)^{-1}\int_{-\infty}^\infty
\exp({-\xi^2\sqrt{|x|}\omega_4+i\xi^3/3})d\xi.$$

For $x\geq -1$, set ${\zeta}(x)=1/(1+x^2)$ and
$\theta(x)=e^{-i/(1+x^2)}\Ai(x)/2$. By  \cite[(7.6.20), Page
215]{Ho1} function $\Ai$ and all its derivatives decay   exponentially when $x$ tends to $+\infty$.   This finishes the proof of Lemma~\ref{prop7.1}.   \end{proof}

The next statement is a standard oscillatory integral type estimates.

\begin{theorem}\label{osc-int}
Suppose that  $\psi \in C^{\infty}({\mathbb R})$ is a real valued  function and that   $u
\in C^{\infty}_c(V)$, where $V$ is a closed subset of
$\mathbb{R}$. Then for each positive integer $l>0$ there exists a positive
constant $C_l$ such that for $\lambda>0$
$$
\left|\int \exp(i\lambda\psi(t))u(t)dt\right| \leq C_l\sum_{k=0}^{l}
\sup_V|u^{(k)}||\psi'|^{k-2l}\lambda^{-l}.
$$
The constant $C_l$ in the above estimate is bounded when $\psi$ stays in a bounded
set in $C^{l+1}(V)$.
\end{theorem}
\begin{proof}
Theorem \ref{osc-int} is a special one-dimensional case of
\cite[Theorem 7.7.1]{Ho1} under additional assumption that function $\psi$ is real valued.
\end{proof}

In the next statement we will describe estimates for the
kernel $K_{w({\A})}(x,y)= w({\A})\delta_y(x)$ of the spectral multiplier  $w(\A)$ which play
crucial role in the proof of our main result.

\begin{proposition}\label{A-est}
Let $w\in C^\infty_c(\R) $ be a smooth function such that $\supp w
\subset [-a, a]$. For all $k\in {\N}$ choose constants $C_{k}>0$ in
such a way that
\begin{eqnarray}\label {con_w}
  |d_x^k w|(x) \leq {C}_{k}a^{-k}
\end{eqnarray}
and $C_k$ do not depend on $a$.

\noindent Then

\begin{enumerate}
\item[A)]
 For all $x\in \R$
and $y$ satisfying $a \geq \min(1, |y|^{-1/2})$,
\begin{equation}\label{abig}
|K_{w({\A})}(x,y)|= |w({\A})\delta_y|(x) \leq C'_{l}\frac{d^{-1}}{(1 + |x -
y|/d)^{l}}\left(1+\frac{|y|}{1+|x|}\right)^{\frac{1}{4}}
\end{equation}
where $d = \max(a^{-1/2}, |y|^{1/2}/a)$ and $C'_{l}$ just depends on
the constants ${C}_{k}$ in \eqref{con_w}  and $l$,  but not on $a,
x$ and $y$.

\item[B)]
 For all $x\in\R$ and $y$ satisfying  $a \leq \min(1,
|y|^{-1/2})$,
\begin{equation}\label{asmall}
|K_{w({\A})}(x,y)|=|w({\A})\delta_y|(x) \leq C''_l  \frac{a}{(1+ a^{2}|x|)^{l}}
(1 + |y|)^{-1/4}(1 + |x|)^{-1/4}
\end{equation}
where $C''_{l}$ just depends on ${C}_{k}$ in \eqref{con_w} and $l$,
but not on $a, x$ and $y$.
\end{enumerate}
\end{proposition}

We shall prove Parts A) and B) separately.

\begin{proof}[Proof of Part A)]
Recall that in Part A) of Proposition  \ref{A-est} we assume that $a
\geq \min(1, |y|^{-1/2})$. We split the proof of Part A) of
Proposition \ref{A-est} into two cases:
\begin{enumerate}
\item[{\bf I)}] $|y|\leq a^4$;
\item[{\bf II)}] $|y| > a^4$.
\end{enumerate}

\bigskip

 \noindent {\bf {Case I}}: $|y|\leq a^4$.
 In fact our  argument in {\bf Case I}
yields a stronger version of inequality \eqref{abig}
 mainly
\begin{equation}\label{abig1}
 |w({\A})\delta_y|(x) \leq C'_{l}\frac{d^{-1}}{(1 + |x -
y|/d)^{l}}.
 \end{equation}
Note that if $|y|\leq a^4$ then immediately, combing the assumption
$a \geq \min(1, |y|^{-1/2})$, we have $a\geq 1$. Put
$$
h(\lambda) = {\mathcal{T}{(w({\A})\delta_y)}}(\lambda) =
w(\lambda)\Ai(y - \lambda).
$$
Next we calculate the Fourier transform of $h$,
\begin{eqnarray*}
{\hat h}(\omega) &=& \int {\hat w}(t) \exp(-i((\omega - t)^3/3 +
y(\omega - t)))dt\\
 &=&
\exp(-i(\omega^3/3 + y\omega)) \int {\hat w}(t)\exp(-i(-t\omega^2 +
t^2\omega - t^3/3 -yt))dt\\
 &=& \exp(-i(\omega^3/3 +
y\omega))g(\omega),
\end{eqnarray*}
where the last equality defines function $g$.
Note that
\begin{eqnarray*}
\widehat{w({\A})\delta_y}(\omega)=\widehat{\Ai * h}(\omega) =
\hat{\Ai}(\omega) \hat{h}(\omega) = e^{-iy \omega}g(\omega).
\end{eqnarray*}
Hence to prove estimate \eqref{abig1} it is enough to show that
\begin{equation}\label{abig2}
|\hat{g}(x) | \le C'_{l}d^{-1}(1 + |x|/d)^{-l}.
\end{equation}
Recall that  $d = \max(|y|^{1/2}/a, a^{-1/2})$. We make the following claim.

\bigskip

{\it Claim.} If $g$ is a function defined above then
there exist constants $C'_{k}$ such that
\begin{equation}\label{g-est}
| g^{(k)}|(\omega) \leq C'_{k}\frac{d^k}{1 + {|\omega^2 +
y|}/{a}}
\end{equation}
for all $k\in
{\N}$ and $\omega \in \R$.

\bigskip

First we observe that estimate \eqref{abig2} and in fact whole {\bf Case I} follows from \eqref{g-est}. Indeed set $\omega_y=\sqrt{\max (0,-y)}$
and note that
$$
\min(|\omega - \omega_y|^2, |\omega + \omega_y|^2) \leq
|\omega - \omega_y||\omega + \omega_y| \le  |\omega^2 + y|
$$
Hence
\begin{eqnarray*}
\int_{\R} \frac{1}{1 + {|\omega^2 + y|}/{a}}d\omega \le \int_{\R}
\left( \frac{1}{1+|\omega-\omega_y|^2/a}+
\frac{1}{1+|\omega+\omega_y|^2/a}\right)   d \omega = Ca^{1/2}.
\end{eqnarray*}
It follows now from estimates  \eqref{g-est} and the relation $d = \max(|y|^{1/2}/a,
a^{-1/2})\geq a^{-1/2}$ that
$$
\|g^{(k)}\|_{1} \le C C'_{k} d^k a^{1/2} \le C C'_{k} d^{k-1}.
$$
Now the estimates \eqref{abig2} and {\bf Case I} are straightforward
consequence of $L^1$ estimates of derivatives of $g$ stated above.
 Hence to finish the proof of {\bf Case I} it is enough to show estimate
\eqref{g-est}.

\bigskip

{\it Proof of Claim} \eqref{g-est}.
Set
$$
\psi(\omega, t) = a^{-1}(\omega^2 + y)t - a^{-2}\omega t^2 +
a^{-3}t^3/3.
$$
 Substituting $t/a$ for $t$ we have
$$
g(\omega) = \int u(t)\exp(i\psi(\omega, t)) dt
$$
where $u(t) = {\hat w}(t/a)/a$.  Now let $\eta$ be a smooth cutoff
function such that $\supp(\eta) \subset [-1, 1]$ and
$\sum_{j\in\mathbb{Z}} \eta(t - j) = 1$ for all $t \in {\mathbb R}$.
Write $u(t)=\sum_{j \in \Z} u_j(t)=\sum_{j \in \Z}u(t)\eta(t - j)$. Decompose $g$ as
$$
g(\omega)=\sum_{j \in \Z}g_j(\omega)= \sum_{j \in \Z}\int u_j(t) \exp(i\psi(\omega, t)) dt.
$$
Now to prove \eqref{g-est} for $k=0$
it is clearly enough to show that for some   natural  $N_1\ge 2$ and  every  $N_2 \ge 1$
\begin{eqnarray} \label{g_j-est}
|g_j|(\omega)\leq \frac{C_{N_1,N_2}(1+|j|)^{-N_1}}{(1 + |\omega^2 + y|/a)^{N_2}}.
\end{eqnarray}
with constant $C_{N_1,N_2}$ independent of $j$. In fact in the case $k=0$ it is  enough to
consider term $1 + |\omega^2 + y|/a$ instead of $(1 + |\omega^2 + y|/a)^{N_2}$,
but we will have to verify  a bit stronger estimate when we consider the case $k>0$,
see \eqref{der-coef1} below.
Next, assumption \eqref{con_w} on function $w$ and the fact that $\supp\ u_j
\subset [j-1, j+1]$ yields
$$
|d_t^{k}u_j(t)|\leq C_{k,N}(1+|j|)^{-N}
$$
for all $k\in {\N}$.
Hence
$$
 |g_j(\omega)|\leq \int_{j-1}^{j+1} |u_j(t)|dt\leq C_{N}(1+|j|)^{-N} \quad .
$$
Now if $|\omega^2 + y|/a\leq 32(1+|j|)^2$ then \eqref{g_j-est} is a straightforward
consequence of the above estimate so we can assume further on that
$|\omega^2 + y|/a> 32(1+|j|)^2$.

If this is the case we want to estimate $g_j(\omega)$ as an oscillatory integral.
When $|\omega^2 + y|/a> 32(1+|j|)^2$, the following inequalities
hold for all $t\in [j-1,j+1]$,
\begin{equation}\label{coef1}
\left|{\frac{\omega t}{a(\omega^2 + y)}}\right| < 1/4 \quad \mbox{and} \quad
\left|{\frac{t^2}{a^2(\omega^2 + y)}}\right| < 1/4.
\end{equation}
Indeed, since $a \geq 1$ we have
$$
\left|{\frac{t^2}{a^2(\omega^2 + y)}}\right| =\frac{ a^{-3}t^2}{|(\omega^2 +
y)/a|} \leq \frac{(1+|j|)^2}{|(\omega^2 + y)/a|}<1/4.
$$
When $|\omega^2 + y| \geq \omega^2/2$ then
$$
\left|{\frac{\omega t}{a(\omega^2 + y)}}\right|^2 \leq {\frac{|\omega|^2}{|\omega^2 + y|}}
{\frac{(1+|j|)^2}{ |(\omega^2 + y)/a|}} <1/16.$$
 When $|\omega^2 + y| <
\omega^2/2$, then $|y|
> \omega^2/2$, $|\omega| \leq 2|y|^{1/2}$ and
$$
\left|{\frac{\omega t}{ a(\omega^2 + y)}}\right| \leq {\frac{2|y|^{1/2}
(1+|j|)}{a^2(|\omega^2 + y|/a)}}<1/4
$$
where we used inequality $|y|^{1/2}/a^2 \leq 1$. These calculations verify  (\ref{coef1}).

 Write $\psi(\omega, t) = a^{-1}(\omega^2 + y)\psi_1(\omega,
t)$ where
$$
\psi_1(\omega, t) = t - {\frac{\omega} {a(\omega^2 + y)}}t^2
+ {\frac{t^3}{ 3a^2(\omega^2 + y)}}.$$
We have
$$
\partial_t\psi_1(\omega, t) = 1 - 2{\frac{\omega t}{a(\omega^2 + y)}}
+ {\frac{t^2}{ a^2(\omega^2 + y)}}.
$$
Thus by \eqref{coef1}  $\partial_t\psi_1(\omega,
t)
> 1/4$ and all higher derivatives of $\psi_1$ are bounded. Substituting $\psi=\psi_1$, $u=u_j$ and
$\lambda =a^{-1}(\omega^2 + y)$ in  Theorem~\ref{osc-int} yields
estimate~\eqref{g_j-est}. This proves \eqref{g-est} for $k=0$.

\medskip

To
handle $k>0$ note that $\partial_\omega \psi(\omega, t) =
2a^{-1}\omega t - a^{-2} t^2$ and
$$
\partial_\omega^k \exp(i\psi(\omega, t)) =P_k(\omega, t)
\exp(i\psi(\omega, t))
$$
where $P_k(\omega, t)$ is a polynomial such that
$$
|\partial_t^l P_k(\omega, t)| \leq C_{k, l}(|\omega|^ka^{-k} +
a^{-2k} + a^{-k/2})(1+|t|)^{2k}
$$
for all $t \in \R$. In fact, $P_1 = i(2a^{-1}\omega t - a^{-2}
t^2)$, $P_{k+1} = P_1P_k + \partial_\omega P_k$ and one can
inductively prove that $P_k = \sum_{l, j \in N_k}b_{k,l,j}\omega^l
(t/a)^j$ where $N_k$ is set of points with integer coordinates in
the triangle with vertices $(k,k)$, $(0,2k)$, $(0, k/2)$.  To see
that $(i, j)$ is above or on line trough $(k, k)$ and $(0, k/2)$
assign to $\omega^l t^j$ degree $-l/2 + j$ and note that minimal
degree of term in $P_{k+1}$ is bigger by $1/2$ then minimal degree
of term in $P_{k}$.  Considering normal degree $l + j$ we see that
$(l, j)$ is below line trough $(k, k)$ and $(0, 2k)$ which shows
that indeed $P_k$ is of prescribed form.  Now, we estimate each term
of $P_k$ separately using inequality between arithmetic and geometric
mean.

Next
\begin{eqnarray*}
d_\omega^k g_j(\omega)&=& \int_{j-1}^{j+1} u_j(t)
\partial_\omega^k\exp(i\psi(\omega, t))dt
\\
&=& \int_{j-1}^{j+1} P_k(\omega, t)u_j(t)\exp(i\psi(\omega, t))dt.
\end{eqnarray*}

Repeating the argument which we use above to prove \eqref{g_j-est} with $u_j$ replaced by   $P_k(\omega, t)u_j$ yields
$$
|d_\omega^k g_j|(\omega)| \leq C'_{k}(|\omega|^ka^{-k} +
a^{-2k} + a^{-k/2})\frac{(1+|j|)^{-2}}{(1 + |\omega^2 + y|/a)^{N_2}}.
$$
Since $a^{-1/2} \leq d$ and $a \geq 1$ so  $a^{-2k} \leq a^{-k/2} \leq
d^k$. Hence if
\begin{equation}\label{der-coef1}
(|\omega|/a)(1 + |\omega^2 + y|/a)^{-1/2} \leq Cd
\end{equation}
then claim (\ref{g-est}) is satisfied for all $k$.

To see that (\ref{der-coef1}) holds  we consider two cases.  When
$|\omega^2 + y| \geq \omega^2/2$, then
$$
|\omega|a^{-1}(1 + |\omega^2 + y|/a)^{-1/2} \leq
|\omega|a^{-1}(\omega^2/(2a))^{-1/2} = 2^{1/2}a^{-1/2}.$$
When $|\omega^2 + y| < \omega^2/2$, then $|\omega| \leq 2|y|^{1/2}$
and
$$
|\omega|a^{-1} \leq 2|y|^{1/2}/a \leq 2d
$$
so indeed (\ref{der-coef1}) holds. 
This ends the proof of estimates \eqref{g-est} and   {\bf Case I}.

\bigskip

\noindent {\bf {Case II}}: $|y| > a^4$. In this case, $|y|\geq 1$,
$|y|\geq a^2$ so  $d =|y|^{1/2}/a\geq 1$.  Recall that  by
Lemma~\ref{poz}
$$
(w({\A})\delta_y)(x)=\int w(s)\Ai(y-s)\Ai(x- s)ds.
$$
We further split {\bf {Case II}} into four sub-cases.

\medskip

{\it Case (i)}:  $y < -a - 1$ and $x < - a - 1$ then  by equality
(\ref{ai-asymp})
\begin{eqnarray*}
(w({\A})\delta_y)(x)
&=& \int \exp\Big[i({\zeta}(y - s) + {\zeta}(x - s))\Big]w_1(s)\theta(x - s)ds\\
&&+\int\exp\Big[i({\zeta}(y - s) - {\zeta}(x -
s))\Big]w_1(s){\overline \theta}(x -
s)ds\\
&& + \int\exp\Big[-i({\zeta}(y - s) + {\zeta}(x -
s))\Big]w_2(s){\overline \theta}(x -
s)ds\\
&& + \int\exp\Big[-i({\zeta}(y - s) - {\zeta}(x - s))\Big]w_2(s)\theta(x - s)ds\\
& =& I_1 + I_2 + I_3 + I_4,
\end{eqnarray*}
where $w_1(s) = w(s)\theta(y - s)$ and $w_2(s) = w(s){\overline
\theta}(y - s)$.

 We will only estimate   $I_2$ because the  proofs
for the other integrals are similar. We write
\begin{eqnarray*}
I_2 &=& \int \exp\Big[i({\zeta}(y - s) - {\zeta}(x -
s))\Big]w_1(s){\overline \theta}(x - s)ds
\\
&=& a\int \exp\Big[i({\zeta}(y - at) - {\zeta}(x - at))\Big]u(t,
x)dt
\end{eqnarray*}
where $u(t, x) = w(at)\theta(y-at){\overline \theta}(x - at)$. We
estimate $I_2$ again using the approach of oscillatory integrals.

\medskip

Next we consider three ranges of variable $x,y\in \R$: $|x|> 2|y|$,
$|y|/2\leq |x|\leq 2|y|$ and $|x| \le  |y|/2$. If $|x|> 2|y|$ then
by Lemma~\ref{prop7.1} and assumption  \eqref{con_w},
\begin{eqnarray*}
|\partial_t^{k}u(t, x)| &\leq& C'_k (1+|x-at|)^{-1/4}(1+|y-at|)^{-1/4}
\leq C'_k|y|^{-1/2}
\end{eqnarray*}
where we use the facts that,  $\supp(u(\cdot, x)) \subset [-1, 1]$
so we can assume that $|t| \le 1$; and that if $a\leq 2$ then
$|y|>1+a\geq3a/2$, otherwise  if $a>2$ then  $|y|>a^4>2a$. Now  if $|x -
y|a|y|^{-1/2} \leq 1$ then
$$
I_2 \leq Ca|y|^{-1/2}\leq C a|y|^{-1/2}C'_l(1 + |x -
y|a|y|^{-1/2})^{-l} \leq C'_l d^{-1}(1 + |x - y|/d)^{-l}.
$$
Thus we can assume $|x - y|a|y|^{-1/2} > 1$. In addition in the
consider range of $x,y$ we have ${\zeta}(z)=2(-z)^{3/2}/3$
 so
\begin{eqnarray*}
\partial_t({\zeta}(y - at) - {\zeta}(x - at))
= a((-y + at)^{1/2} - (-x  + at)^{1/2}).
\end{eqnarray*}
Then absolute value of the derivative is bounded below by $c(|x -
y|a|y|^{-1/2})^{1/2}$ since $|y|\geq a$, $a\geq |y|^{-1/2}$ and
$|x|\sim |x-y|$. Also note that by $|y| \geq a^4$, if $|x -
y|a|y|^{-1/2}
> 1$, $|x - y| > |y|^{1/2}/a \geq a$. Next we observe that higher
derivatives of the phase function are uniformly bounded by constant
$C(|x -y|a|y|^{-1/2})^{1/2}$. The lower bounds for derivative
verified above shows that  we can  choose $\lambda \sim (|x
-y|a|y|^{-1/2})^{1/2}$ in a  such way that $\partial_t({\zeta}(y -
at) - {\zeta}(x - at))/ \lambda \ge 1$. Substituting $\psi =
({\zeta}(y - at) - {\zeta}(x - at))/ \lambda$ in
Theorem~\ref{osc-int} yields
\begin{eqnarray*}
I_2 &\leq& a|y|^{-1/2}C'_l(1 + |x - y|a|y|^{-1/2})^{-l} \leq C'_l
d^{-1}(1 + |x - y|/d)^{-l}.
\end{eqnarray*}

\medskip
 For the range
$|y|/2\leq |x|\leq 2|y|$, we use the similar argument as above. For
the function $u$, we still have
\begin{eqnarray*}
|\partial_t^{k}u(t, x)| &\leq& C'_k
(1+|x-at|)^{-1/4}(1+|y-at|)^{-1/4} \leq C'_k|y|^{-1/2}.
\end{eqnarray*}
However we have to modify required estimates for the phase function.
As before, we can assume $|x - y|a|y|^{-1/2} > 1$ and in the
consider range of $x,y$ we have
\begin{eqnarray*}
\partial_t({\zeta}(y - at) - {\zeta}(x - at))
= a((-y + at)^{1/2} - (-x  + at)^{1/2}).
\end{eqnarray*}
Then absolute value of the derivative is bounded below by $c|x -
y|a|y|^{-1/2}$ since $|y|\geq a$, $a\geq |y|^{-1/2}$ and $|x|\sim
|y|$. Next we observe that higher derivatives of the phase function
are uniformly bounded by  constant $C|x -y|a|y|^{-1/2}$ since
$|y|\geq a$, $|x|\sim |y|$ and $|x - y|a|y|^{-1/2} > 1$. The lower
bounds for derivative verified above shows that  we can choose
$\lambda \sim |x -y|a|y|^{-1/2}$ in a  such way that
$\partial_t({\zeta}(y - at) - {\zeta}(x - at))/ \lambda \ge 1$.
Substituting $\psi = ({\zeta}(y - at) - {\zeta}(x - at))/ \lambda$
in Theorem~\ref{osc-int} yields
\begin{eqnarray*}
I_2 &\leq& a|y|^{-1/2}C'_l(1 + |x - y|a|y|^{-1/2})^{-l} \leq C'_l
d^{-1}(1 + |x - y|/d)^{-l}.
\end{eqnarray*}

\medskip
For the range $|x|\leq |y|/2$ we use the same argument as before
with the same phase function  $\psi = ({\zeta}(y - at) - {\zeta}(x -
at))/ \lambda$. However we have to modify required estimates for
$\lambda$ and the function $u$.

 If $|x|\leq |y|/2$, $a\geq 2$ and $|x|< 2a$, then note that
 $a\geq 2> \min\{1, |x|^{-1/2}\}$, $|x|<2a\leq a^4$ and
 $(w({\A})\delta_y)(x)=(w({\A})\delta_x)(y)$. By the result of
{\bf {Case I}}, we have that for $d'=\max\{a^{-1/2},
|x|^{1/2}/a\}\leq 1$,
\begin{eqnarray*}
(w({\A})\delta_y)(x)=(w({\A})\delta_x)(y)\leq
C'_ld'^{-1}(1+|x-y|/d')^{-l-1}.
\end{eqnarray*}
Note that $d\geq 1$, $d'\leq 1$, $a\geq 2$, $|y|\geq 1$ and
$|x-y|\sim |y|$, then
\begin{eqnarray*}
(w({\A})\delta_y)(x)\leq C'_la |x-y|^{-1}(1+|x-y|/d')^{-l} \leq
C'_la|y|^{-1/2}(1+|x-y|/d)^{-l}.
\end{eqnarray*}

   If $|x|\leq |y|/2$, $a\geq 2$ and $|x|\geq 2a$, we get
$$
|\partial_t^{k}u(t, x)| \leq C'_k|y|^{-1/4}.
$$
and $\supp(u(\cdot, x)) \subset [-1, 1]$. Note that $|x|\leq |y|/2$
implies that $|x-y|\sim |y|$. Thus the absolute value of the
derivative of the phase function is bounded below by $ca|y|^{1/2}$
and the higher derivative of the phase function is bounded by
$Ca|y|^{1/2}$. Then as before, choosing $\lambda \sim a|y|^{1/2}$,
by Theorem~\ref{osc-int}
\begin{eqnarray*}
I_2 &\leq& a|y|^{-1/4}C'_l(a|y|^{1/2})^{-l-1}\\
&\leq&C'_l a|y|^{-1/2}(a|y|^{1/2})^{-l}a^{-1}|y|^{-1/4}\\
&\leq&C'_l a|y|^{-1/2}(1+a|y|^{1/2})^{-l}\\
&\leq&C'_l a|y|^{-1/2}(1+a|x-y||y|^{-1/2})^{-l}\\
&\leq&C'_l d^{-1}(1+|x-y|/d)^{-l}.
\end{eqnarray*}
If $|x|\leq |y|/2$ and $a< 2$, we get $|x|\geq 3a/2$ and thus
$$
|\partial_t^{k}u(t, x)| \leq C'_k|y|^{-1/4}(1+|x|)^{-1/4}
$$
and $\supp(u(\cdot, x)) \subset [-1, 1]$. Note that $|x|\leq |y|/2$
implies that $|x-y|\sim |y|$. Thus the absolute value of the
derivative of the phase function is bounded below by $ca|y|^{1/2}$
and the higher derivative of the phase function is bounded by
$Ca|y|^{1/2}$. Then as before, choosing $\lambda \sim a|y|^{1/2}$,
by Theorem~\ref{osc-int}
\begin{eqnarray*}
I_2 \leq a|y|^{-1/4}(1+|x|)^{-1/4}C'_l(a|y|^{1/2})^{-l}
&\leq&C'_l a|y|^{-1/2}(1+a|x-y||y|^{-1/2})^{-l}\frac{|y|^{1/4}}{(1+|x|)^{1/4}}\\
&\leq&C'_l
d^{-1}(1+|x-y|/d)^{-l}\Big(\frac{|y|}{1+|x|}\Big)^{\frac{1}{4}}.
\end{eqnarray*}

\medskip

{\it Case (ii)}: $y \geq -a - 1$ and $x < - a - 1$. For
$|s|\leq a$, when $a<2$, $y-s\geq -a-1-a\geq-5$ and so all
derivatives of $\Ai(y-s)$ are exponential decay. Thus for all
$k,l,l'\in {\N}$
\begin{eqnarray}
|\partial^k_s(w(s)\Ai(y-s))|&\leq& \sum_{m\leq k}C_{m,l}a^{-m}(1 +
a^{-1}|s|)^{-l}(1+|y-s|)^{-l'}\nonumber\\
&\leq& C'_{k,l}a^{-k}(1 +
a^{-1}|s|)^{-l}(1+|y|)^{-l'}\label{C.11Aprl_wAi1}.
\end{eqnarray}
When $a\geq 2$, by $|y|>a^4$, $y$ must be bigger than $0$ and
$y-s>a^4-a>2a>0$. Thus all derivatives of $\Ai(y-s)$ are exponential
decay. Then for all $k,l,l'\in {\N}$
\begin{eqnarray}
 |\partial^k_s(w(s)\Ai(y-s))|&\leq&
\sum_{m\leq k}C_{m,l}a^{-m}(1 +
a^{-1}|s|)^{-l}(1+|y-s|)^{-(l'+k-m)}\nonumber\\
&\leq& \sum_{m\leq k}C_{m,l}a^{-m}(1+|y-s|)^{m-k}(1 +
a^{-1}|s|)^{-l}(1+|y-s|)^{-l'}\nonumber\\
&\leq& C'_{k,l}a^{-k}(1 +
a^{-1}|s|)^{-l}(1+|y|)^{-l'}\label{C.11Aprl_wAi2}.
\end{eqnarray}
When $|y|>|x|/2$, by estimates~\eqref{C.11Aprl_wAi1},
\eqref{C.11Aprl_wAi2} and $d\geq 1$
\begin{eqnarray*}
|(w({\A})\delta_y)(x)| &=&|\int w(s)\Ai(y-s)\Ai(x - s)ds|\\
&\leq&2a C(1+|y|)^{-l'}\\
&\leq&Ca |y|^{-1/2}(1+|x-y|)^{-l}\\
&\leq&Cd^{-1}(1+|x-y|/d)^{-l}.
\end{eqnarray*}
When $|y|\leq |x|/2$  by (\ref{ai-asymp})
\begin{eqnarray}\label{likeosc}
(w({\A})\delta_y)(x) &=&\int w(s)\Ai(y-s)\Ai(x - s)ds\\  \nonumber
&=&\int h(s)\theta(x-s)e^{i(2((-x  + s)^{3/2})/3)}ds\\  \nonumber
&&+\int h(s)\bar\theta(x-s)e^{-i(2((-x  + s)^{3/2})/3)}ds\\
\nonumber &=&a\int h(at)\theta(x-at)e^{i(2((-x  + at)^{3/2})/3)}dt\\
\nonumber &&+a\int h(at)\bar\theta(x-at)e^{-i(2((-x  +
at)^{3/2})/3)}dt.  \nonumber
\end{eqnarray}
For function $h(at)\theta(x-at)$ or $h(at)\bar\theta(x-at)$, by
estimates~\eqref{C.11Aprl_wAi1}, \eqref{C.11Aprl_wAi2} and
Lemma~\ref{prop7.1}, all the derivatives are bounded by $
C'_k|y|^{-1/2}. $ For the phase function $2((-x  + at)^{3/2})/3$,
since $|y|\geq a$ and $|x|\geq2|y|\geq 2a$,   $\partial_t(2((-x +
at)^{3/2})/3) = a((-x + at)^{1/2})$ is bounded below by
$ca|x|^{1/2}$ and all higher derivatives are bounded by
$Ca|x|^{1/2}$. Then  by Theorem~\ref{osc-int} and $a|x|^{1/2}\geq
a^{1/2}|y|^{-1/4}|x-y|^{1/2}$,
\begin{eqnarray*}
(w({\A})\delta_y)(x)&\leq&aC'_l|y|^{-1/2}(a|x|^{1/2})^{-l}\\
&\leq&C'_la|y|^{-1/2}(1+a|y|^{-1/2}|x-y|)^{-l/2}\\
&\leq&C'_ld^{-1}(1+|x-y|/d)^{-l/2}.
\end{eqnarray*}

\medskip


 For  {\it Case (iii)}: $y
\geq -a - 1, x \geq - a - 1$, when $|y|>|x|/2$, the proof is similar
to that in the situation $|y|>|x|/2$ of {\it Case (ii)}; when
$|y|\leq |x|/2$, because $|x|\geq 2|y|>2a^4$, the estimate
\eqref{abig} follows from that both $\Ai(x-s)$ and $\Ai(y-s)$ decay
exponentially .

For {\it Case (iv)}: $y < -a - 1, x \geq - a - 1$, when $|x|>|y|/2$,
the proof is similar to that in the situation $|y|>|x|/2$ of {\it
Case (ii)}; when $|x|\leq |y|/2$, $a>2$ and $|x|\leq 2a$, the proof
is similar to that in the situation $|x|\leq |y|/2$, $a\geq 2$ and
$|x|\leq 2a$ of {\it Case (i)}; for the other situations, the proof
is similar to that in the situation $|y|\leq |x|/2$ of {\it Case
(ii)}.

\end{proof}

Next we discuss  the proof of Part B of Proposition  \ref{A-est}.

\begin{proof}[Proof of Part B)]
Recall that in Part B)  of Proposition  \ref{A-est} we assume that  $a \leq \min(1, |y|^{-1/2})$. It is not difficult to notice that for $x \ge -2 $ estimate \eqref{asmall} is straightforward
consequence of exponential decay of the Airy function for positive argument.
Hence we only consider $x \le -2$.

Note that if
$|s|\leq a$ then  $|\Ai(y-s)|\leq C (1+|y|)^{-1/4}$ and otherwise $w(s)=0$.
Next, in the considered case
$a\leq 1$ and $ |y|^{1/2}\leq a^{-1}$ so it follows from  Lemma \ref {prop7.1}  that
\begin{eqnarray*}
|\partial_s\Ai(y-s)|&\leq& C (1+|y-s|^{1/2})(1+|y-s|)^{-1/4}\leq
C\max\{1, |y|^{1/2}\}(1+|y|)^{-1/4}\\
&\leq&Ca^{-1}(1+|y|)^{-1/4}
\end{eqnarray*}
for all $|s|\leq a$.
Inductively, using 
the defining relation   $\Ai''(x)=x\Ai(x)$,
we get
$$
|\partial^k_s\Ai(y-s)|\leq Ca^{-k}(1+|y|)^{-1/4}\quad  \forall_{|s|\leq a}.
$$
Now it follows from assumptions on $w$ (that is $\supp w \subset [-a,a]$ and
 \eqref{con_w}) that
 the function $h(s)=w(s)\Ai(y-s)$ satisfies the estimate
$$
 |\partial^k_s(h(s))|
\leq C'_{k,l}a^{-k}(1 + a^{-1}|s|)^{-l}(1+|y|)^{-1/4}  .
$$
 Using the above inequality, then writing 
\begin{eqnarray*}
w(A)\delta_y(x)=\int w(s)Ai(y-s)  Ai(x-s) ds
=\int w(s)Ai(y-s)\theta(x-s) e^{i(2(-x+s)^{3/2}/3)} ds\\+ \int w(s)Ai(y-s)\bar{\theta}(x-s) e^{-i(2(-x+s)^{3/2}/3)} ds
\end{eqnarray*}
and setting $u(t)=w(s)Ai(y-s)\theta(x-s)$ or $u(t)= w(s)Ai(y-s)\bar{\theta}(x-s)$ 
  yield
 estimate \eqref{asmall} by oscillatory integrals argument of Theorem \ref{osc-int}.

\end{proof}


\section{Proof of Theorem \ref{main}}\label{sec6}


This section is entirely devoted to the proof of Theorem \ref{main}.
Because the argument is rather complex we divide it into several
steps formulated as separate statements. First in Lemmas
\ref{G1-neg} and \ref{G2-est} we split the multiplier $F$ into large
but analytic part $G$ and small but rough part $H$. Next in Lemma
\ref{G2-L1} we estimate $L^1\to L^1 $ norm of ${H}(\LL/\lambda)$ and in
Lemma~\ref{G2-L4/3} we discuss  $L^{4/3} \to L^{4/3}  $ norm of the same operator.
To obtain  these $L^{4/3}$ estimates we ``~interpolate" between
$L^1$ and $L^2$ but the interpolation argument is not standard.
Later we use a similar interpolation trick in the proof of $L^{4/3}$
estimates for ${G}(\LL/\lambda)$ in Lemma~\ref{Lam3/4}. In Lemma
\ref{G1-decomp} we describe further wavelet like decomposition of
the nice (analytic) part $G$. This decomposition allow us to apply
estimates for the Airy operator which we obtained in Section
\ref{sec5} to study the multiplier ${G}(\LL/\lambda)$, see Lemma
\ref{gkm-est} below. In fact a bit earlier in Lemma \ref{AL}, based
on the finite speed propagation property for the wave equation,  we
show that in crucial part of our argument we can replace multiplier
${G}(\LL/\lambda)$ by the multipliers corresponding to the Airy
operator.

\medskip

Our aim is to investigate bounds for  $\|F(\LL/\lambda)\|_{p \to p}$ for $1\leq p\leq 2$ and for large $\lambda$.
For $p\geq 2$, we use duality  and for small $\lambda$ any required estimates hold, see the discussion before Lemma 3.5.
We want to estimate  the kernel $F(\LL/\lambda)\delta_y$, where $F\in H^s$,
$s > 1/2 $ and $\supp F \subset [1/2, 1]$.  By Proposition \ref{I/4} 
we know that $\|F(\LL/\lambda)I_{\lambda/4}\|_{p\to p } \leq
C\|F\|_{H^{1/2+\epsilon}}$ for all $1\le p \le 2$ so we can assume that   $|y| > \lambda/4$.
It follows also from the obvious symmetry of the considered operator $\LL$ that we can also assume  that $y> \lambda/4>0$ without loss of generality.
In addition it follows from Lemma~\ref{2I}
that in the proof  we only need to estimate the norm of the restricted operator 
$\|I_{2\lambda}F(\LL/\lambda)I_{2\lambda}\|_{p\to p}$ for $1\leq p\leq 2$.

\medskip

We write
$$
F(\LL/\lambda) = {\tilde F}(\sqrt{{\LL}/\lambda}) = {\tilde
F}(\lambda^{-1/2}\sqrt{{\LL}})
$$
where ${\tilde F}(x) = F(x^2)$.  Let $\psi$ be a function such
that ${\hat \psi}$ is smooth, $\supp {\hat \psi} \subset [-1, 1]$,
${\hat \psi} = 1$ on $[-1/2, 1/2]$, $0 \leq {\hat \psi} \leq 1$,
and ${\hat \psi}$ is symmetric.  Next for $h>0$ we set  $\psi_h(x) = h\psi(hx)$ and we define
function ${\tilde G}$ in the following way
$$
{\tilde G} = {\tilde F}*\psi_{\lambda^{3/2}/6}.
$$
Note that
$$
\|{\tilde G}\|_{H^s} \leq C \|{\tilde F}\|_{H^s} \leq C\|F\|_{H^s}
$$
and
$$
\|{\tilde F} - {\tilde G}\|_{H^s}
\leq C\|{\tilde F}\|_{H^s} \leq C\|F\|_{H^s}
$$
Now we define functions $F$ and $G$ by the following formula
\begin{equation}\label{GH}
{G}(x) = {\tilde G}(\sqrt{x}) \quad \mbox{and} \quad  {H} = F - {G}.
\end{equation}

Note that functions $G$ and $H$ depend on choice of $\lambda$.
In the rest of this section $\lambda$ is treated as fixed large constant.
In Lemmas \ref{G1-neg} and \ref{G2-est} below we derive some straightforward
differentiability properties of $G$ and $H$ which we use to estimate 
 ``tail" parts of spectral multipliers of $G(\LL/\lambda )$ and $H(\LL/\lambda)$.

\begin{lemma}\label{G1-neg}
Assume that $\lambda >1$,  $\supp F  \subset [1/2, 1]$, $F\in L^2(\R)$ and let
${G}$ be the function corresponding to $F$ and $\lambda$ defined by \eqref{GH}.
Then ${G}$ can be extended to an entire analytic function and there exists a constant
$C>0$ such that
$$
|{G}(z)| \leq
C\lambda^{3/2}\exp(\lambda^{3/2}|z|^{1/2}/6)\|F\|_{L^2}
$$
for all $z\in\C$.
\end{lemma}
\begin{proof}
Note that since $\supp({\hat \psi}) \subset [-1, 1]$ and
$\int |{\hat \psi}| \leq 1$, $\psi$ is an entire
analytic function satisfying
$$
|\psi(z)| \leq \exp(|\Im(z)|).
$$
Consequently ${\tilde G}$ is an entire function and
$$
|{\tilde G}|(z)
\leq C\lambda^{3/2}\exp(\lambda^{3/2}|\Im(z)|/6)\|F\|_{L^2}
$$
(the last inequality follows since the $L^2$ norms of $F$ and ${\tilde F}$ are comparable).  Note that both ${\tilde F}$ and $\psi$ are symmetric so
${\tilde G}$ is also symmetric. Hence ${G}$ is a well defined entire
function.  Thus
\begin{eqnarray*}
|{G}(z)| = |{\tilde G}(\sqrt{z})| &\leq&
C\lambda^{3/2}\exp(\lambda^{3/2}|\Im(\sqrt{z})|/6)\|F\|_{L^2}
\\
&\leq& C\lambda^{3/2}\exp(\lambda^{3/2}|z|^{1/2}/6)\|F\|_{L^2}.
\end{eqnarray*}
This ends the proof of Lemma \ref{G1-neg}.
\end{proof}

In the next lemma we describe  the behavior and of $L^2$ norm of
function $H$ depending on $\lambda$ and we notice that outside the
support of $F$ function $H$ decays rapidly.

\begin{lemma}\label{G2-est}
Assume that $\lambda >1$, $\supp F  \subset [1/2, 1]$, $s\geq 0$, $F\in
H^s(\R)$ and let ${G}$ and ${H}$ be the functions corresponding to $F$ and $\lambda$ defined by \eqref{GH}. Then there exists a constant $C$ such that
$$
\|{H}\|_{L^2(\R_+)} \leq C(\lambda^{3/2})^{-s}\|F\|_{H^s}.
$$
In addition there exists a constant $C$ such that,
$$
\sup_{0\leq x\leq 1/4}\big(|H(x^2)|+|d_x(H(x^2))|+|d_x^2(H(x^2))|\big)\leq C\|F\|_{L^2}
$$
and for every $l\in \N$ there exists a constant $C_l$ such that
$$
\sup_{x>2}\left(|{H}(x)| + |xd_x {H}(x)| + |x^2d_x^2
{H}(x)|\right)x^l \leq C_l \lambda^{-l}\|F\|_{L^2}
$$
for all $\lambda >1$.
\end{lemma}
\begin{proof}
We have $\|{\tilde F}\|_{H^s} \leq C\|F\|_{H^s}$, so $\|{\tilde
G}\|_{H^s} \leq C\|F\|_{H^s}$.  Since Fourier transform of
$\psi_{\lambda^{{3}/{2}}/6}$ is $1$ on $[-\lambda^{3/2}/12,
\lambda^{3/2}/12]$, we have
$$
\|{\tilde F} - {\tilde G}\|_{L^2} \leq
(\lambda^{3/2}/12)^{-s}\|{\tilde F}\|_{H^s} \leq
C(\lambda^{3/2})^{-s}\|F\|_{H^s}.
$$
Changing variables yields
\begin{eqnarray}\label{l2-subst}
\int_{0}^4|{H}|^2(x)dx &=& \int_{0}^4|{\tilde F} - {\tilde
G}|^2(\sqrt{x})dx \leq 4 \|{\tilde F} - {\tilde G}\|_{L^2}^2\\
&\leq& C(\lambda^{3/2})^{-s}\|F\|_{H^s}.\nonumber
\end{eqnarray}
Since $\psi$ belongs to set of Schwartz class functions and
$\supp{\tilde F} \subset [-1, -1/\sqrt{2}] \cup [1/\sqrt{2},1]$ so for every $l
\in \N$ there exists constant $C_l$ such that
$$
|{H}( x^2)|=|{\tilde F} - {\tilde G}|(x) \leq
C_l(\lambda^{3/2} x)^{-l}\|F\|_{L^2}
$$
for all $x > 2$. Thus for $l'$ large enough
$$
\int_4^\infty |{H}(x)|^2 dx
\leq C\lambda^{-l'}\|F\|^2_{L^2} \le (\lambda^{3/2})^{-2s}\|F\|^2_{L^2}
$$
which together with (\ref{l2-subst}) gives the first estimate for $L^2$ norm
of ${H}$.

To show the second estimate we note that $\supp{\tilde F}$ and the interval $[0,1/4]$ are 
disjoint so for any $l \in  \N$ there exists a constant $C_{l}$ such that for all $x \leq 1/4$
$$
|d_x^{l}(H(x^2))|=|d_x^{l}({\tilde F} - {\tilde G})|(x)= |d_x^{l}{\tilde G}|(x) \leq
C\|F\|_{L^2}
$$
for all $\lambda\geq 1$.

To show the third estimate we note that for any $l',l''\in \N$ there
exists constant $C=C_{l',l''}$ such that for all $x \geq 2^{1/2}$
$$
|d_x^{l'}({\tilde F} - {\tilde G})|(x)= |d_x^{l'}{\tilde G}|(x) \leq
C(\lambda^{3/2}x)^{-l''}\|F\|_{L^2}.
$$
Now the third estimate follows from the above inequality.

\end{proof}

\begin{lemma}\label{G2-L1}
Let $\LL$ be an anharmonic oscillator defined by \eqref{op} and assume next
that   $s \geq 1/2 + 1/6$, $\supp F  \subset [1/2, 1]$, $F\in H^s(\R)$.
Now if ${H}$ is the function  corresponding to $F$ and~$\lambda$ defined by \eqref{GH} then for all $\lambda>1$
$$
\|{H}(\LL/\lambda)\delta_y\|_{L^1(I_{2\lambda})} \leq C\|F\|_{H^s}
$$
where $I_{2\lambda}= [-2\lambda, 2\lambda]$.
\end{lemma}
\begin{proof} Consider function $\nu \in  C_c^\infty(\R)$ such that $\nu (x)=1$ for all
$x\in [1/8, 2]$ and $\supp \nu \subset [1/16,4]$.  The second claim
and the third claim of Lemma~\ref{G2-est} show that for $x \geq 2$
and $x<1/8$  function $(1-\nu) {H}$ satisfies assumptions of most
standard  multiplier theorems and the corresponding spectral multiplier $(1-\nu) {H}(\LL/\lambda)$
satisfies estimate of Lemma~\ref{G2-L1}, see e.g \cite[Theorem 3.2]{DOS}.
\cite[Theorem 3.2]{CoS} or \cite[Theorem 2.4]{He1}. To be more
precise we choose a function $\omega \in C_c^\infty(\R)$ such that
$0 \le \omega (x) \le 1 $ and  $\omega (x)=1$ for $x \le 2$ and
$\omega (x)=0$ for $x \ge 4$. For $n\in \N$ put  $\omega_n(x)=\omega(2^{-n}x)-\omega(2^{-n+1}x)$ so that
$$
\omega(x)+\sum_{n=1}^\infty \omega_n(x) =1 \quad \forall_{x\ge 1}.
$$
Now if we set
$$H_{0,\nu}=\omega(1-\nu) {H} \quad \mbox{and}  \quad H_{n,\nu}=\omega_n(1-\nu) {H}$$
 then by Lemma~\ref{G2-est}
$$
\|\delta_{2^{n+2}}H_{n,\nu}\|_{W^{\infty}_s} \le C  2^{-n}  \|F\|_2
$$
for all $s>1/2$ and all $n\in \N$. Recall that ${W^{\infty}_s}$ is $L^\infty$ Sobolev space of order
$s$.   Note also that $\supp \delta_{2^{n+2}}H_{n,\nu} \subset [-1,1].$   Now by \cite[Theorem 3.1]{COSY} (or by the results from \cite{He1, CoS, DOS} mentioned above)
\begin{eqnarray*}
\|(1-\nu) {H}(\LL/\lambda)\|_{1\to1}\le \sum_{n=0}^\infty \|H_{n,\nu}(\LL/\lambda)\|_{1\to1}
\\=\sum_{n=0}^\infty \|\delta_{2^{n+2}}H_{n,\nu}(2^{-n-2}\LL/\lambda)\|_{1\to1}
\le C \sum_{n=0}^\infty 2^{-n}  \|F\|_2.
\end{eqnarray*}

Hence the operator $(1-\nu) {H}(\LL/\lambda)$ is continuous on all
$L^p(\R)$ spaces and it is enough to consider the the multiplier $
\nu {H}(\LL/\lambda)$. Note that $\supp \nu {H} \subset [1/16, 4]$.
Recall that $\lambda_n$ and $\phi_n$ are the  eigenvalues and
corresponding eigenfunctions of the operator $\LL$. Next write
\begin{eqnarray*}
\|\nu {H}(\LL/\lambda)\delta_y\|^2_{L^2} &=& \sum_{n=1}^\infty |\nu
{H}(\lambda_n/\lambda)|^2|{\phi_{n}}(y)|^2
\\
&\leq& C\lambda^{-1/2}\sum_{n=1}^\infty |\nu
{H}(\lambda_{n}/\lambda)|^2.
\end{eqnarray*}
By inequalities \eqref{e2.3} and \eqref{e2.4} of Proposition  \ref{prop1}
 the distance between $\lambda_n$ and $\lambda_{n+1}$  is of order $\lambda_n^{-1/2} \sim \lambda^{-1/2}$. Hence
by Lemmas  \ref{sup-sum} and \ref{G2-est}
\begin{eqnarray}
\sum_{n=1}^\infty |\nu {H}(\lambda_{n}/\lambda)|^2 \leq
C\sum_{k=0}^\infty \sup_{} \{|\nu {H}(x)|^2\colon \, k\lambda^{-3/2} \leq x \leq (k+1)\lambda^{-3/2}   \}\nonumber  \\
\leq
 C'(\lambda^{3/2}\|\nu {H}\|^2_{L^2} + (\lambda^{3/2})^{1 - 2s}\|\nu {H}\|^2_{H^s})
 \label{g2est} \\
\leq C''(\lambda^{3/2})^{1 - 2s}\|F\|^2_{H^s}\nonumber
\end{eqnarray}
Now, $s \geq 1/2 + 1/6$, so $1 - 2s \leq -2/6$ and
$(\lambda^{3/2})^{1 - 2s} \leq (\lambda^{3/2})^{-2/6} = \lambda^{-1/2}$
which means that
$$
\sum_{n=1}^\infty |\nu {H}(\lambda_{n}/\lambda)|^2 \leq
C\lambda^{-1/2}\|F\|^2_{H^s}
$$
which in turn implies
$$
\|\nu {H}(\LL/\lambda)\delta_y\|^2_{L^2} \leq
C_1\lambda^{-1/2}C_2\lambda^{-1/2}\|F\|^2_{H^s} = C
\lambda^{-1}\|F\|^2_{H^s}.
$$
Hence by H\"older inequality
\begin{eqnarray*}
\|\nu {H}(\lambda^{-1}{\LL})\delta_y\|_{L^1(I_{2\lambda})} &\leq&
|I_{2\lambda}|^{1/2}\| \nu {H}(\lambda^{-1}{\LL})\delta_y\|_{L^2}
\\
&\leq& (4\lambda)^{1/2}(C\lambda^{-1/2})\|F\|_{H^s} \leq
C'\|F\|_{H^s}
\end{eqnarray*}
which yields  the claim.

\end{proof}

\begin{lemma}\label{G2-L4/3}
Let $\LL$ be an anharmonic oscillator defined by \eqref{op} and assume
that $s > 1/2 $, $\supp F  \subset [1/2, 1]$, $F\in H^s(\R)$
and let ${H}$ be the function corresponding to $F$ and $\lambda$ defined by \eqref{GH}. Then
$$
\|{H}(\LL/\lambda)\|_{L^{4/3}(I_{2\lambda})\rightarrow
L^{4/3}(I_{2\lambda})} \leq C_s\|F\|_{H^s}
$$
for all $\lambda \ge 1$.
\end{lemma}
\begin{proof}
Similarly as  in Lemma \ref{G2-L1} it is enough to consider the
operator $\nu {H}(\LL/\lambda)$ where  $\supp(\nu {H}) \subset [1/16,
4]$. We recall that
$$
\nu {H}(\LL/\lambda)f = \sum_{\lambda/16 \le \lambda_n \le 4
\lambda} \nu {H}(\lambda_n/\lambda){\phi_{n}}\langle f,
{\phi_{n}}\rangle
$$
Let now  $\eta \in C^\infty_c({\mathbb R})$ be a such function that $0 \leq \eta \leq 1$,
$\supp \eta \subset [-2,-1/2]\cup [1/2,2]$ and
$$\sum_{j\in \Z}\eta(2^{-j} x) = 1$$ for
all $x \ne 0$.
For $x\geq 0$ put
$$\phi_{j,\lambda_n}(x) = \eta(2^{-j}(x - \lambda_n))\phi_{n}(x)
\quad
\mbox{ for $j>0$ and}
\quad
 {\phi}_{0,\lambda_n} = {\phi}_{n} - \sum_{j>0}{\phi}_{j,\lambda_n}.$$
For $x < 0$ we put ${\phi}_{j,\lambda_n}(x) = 0$. Next for all $j
\ge 0$ set
$$
Q_jf = \sum_{\lambda/16 \le \lambda_n \le 4 \lambda} \nu
{H}(\lambda_n/\lambda){{\phi}_{n}}\langle f, {\phi}_{j,
\lambda_n}\rangle.$$ If $\supp(f) \subset I_{2\lambda}$, then
$$
\nu {H}(\LL/\lambda)f = \sum_{2^j\leq 10\lambda} (Q_j + Q'_j)f
$$
where $Q'_j$ is build like $Q_j$ but using parts of ${\phi}_n$ on
$(-\infty,0)$. Next, by estimate \eqref{eq2.30}
$$
|{\phi}_{j, \lambda_n}|(y) \leq C\lambda^{-1/4}2^{-j/4}.
$$
Consequently if $K_{Q_j}(x, y)=Q_j\delta_y(x)$ is the kernel of the operator $Q_j$ then,
\begin{eqnarray*}
\|K_{Q_j}(\cdot \,, y)\|_2^2=\|Q_j\delta_y\|^2_{L^2} &=&
\sum_{\lambda/16 \le \lambda_n \le 4 \lambda} |\nu {H}(\lambda_n/\lambda)|^2|{\phi}_{j, \lambda_n}|^2(y)\|{{\phi}_{n}}\|^2_{L^2}\\
&\leq& C\lambda^{-1/2}2^{-j/2}\sum_{\lambda/16 \le \lambda_n \le 4
\lambda}|\nu {H}(\lambda_n/\lambda)|^2
\end{eqnarray*}
so by \eqref{g2est}
$$
\|Q_j\delta_y\|_{L^1(I_{2\lambda})} \leq
C\lambda^{1/4}2^{-j/4}(\lambda^{-3/2})^{s-1/2}\|F\|_{H^s}.
$$
Thus
$$
\|I_{2\lambda}Q_j\|_{L^1 \to L^1(I_{2\lambda})} \leq
C\lambda^{1/4}2^{-j/4}(\lambda^{-3/2})^{s-1/2}\|F\|_{H^s}.
$$

\bigskip

Next we consider the $L^2$ norm of the operator $Q_j$. Note that
$$
\|{\phi}_{j, \lambda_n}\|_{L^2}^2 \leq \lambda^{-1/2}2^{j/2}
$$
so
\begin{eqnarray*}
\|Q_jf\|^2_{L^2} &=& \sum_{\lambda/16 \le \lambda_n \le 4 \lambda}
|\nu {H}(\lambda_n/\lambda)|^2|\langle f,
{\phi}_{j,\lambda_n}\rangle|^2
\\
&\leq& \|f\|_{L^2}^2\lambda^{-1/2}2^{j/2}\sum_{\lambda/16 \le
\lambda_n \le 4 \lambda} |\nu {H}(\lambda_n/\lambda)|^2
\end{eqnarray*}
and
$$
\|Q_j\|_{L^2\rightarrow L^2} \leq
C\lambda^{-1/4}2^{j/4}(\lambda^{-3/2})^{s-1/2}\|F\|_{H^s}
$$
Now, by interpolation
$$
\|Q_j\|_{L^{4/3}(I_{2\lambda})\rightarrow L^{4/3}(I_{2\lambda})}
\leq C(\lambda^{-3/2})^{s-1/2}\|F\|_{H^s}
$$
Hence
$$
\sum_{2^j\leq 10\lambda}\|Q_j\|_{L^{4/3}(I_{2\lambda})\rightarrow
L^{4/3}(I_{2\lambda})} \leq
C\log(\lambda+1)(\lambda^{-3/2})^{s-1/2}\|F\|_{H^s}
$$
which is bounded when $s>1/2$.  We get estimate for $Q'_j$ by
symmetry which ends the proof.
\end{proof}

We now move to estimates for the part of the multiplier corresponding to the function $G$. 
For any $k\in \Z$, $k\ge 0$ we define a set $\Delta_k \subset \Z$ by the formula
\begin{equation}\label{delta-k}
\Delta_k=\left\{m\in \Z \colon 0 \le m \le 2^{k+2}\right\}.
\end{equation}
In the next lemma we describe useful wavelet like decomposition of the function $G$.

\begin{lemma}\label{G1-decomp}
Let $\lambda >1$ and  ${G}$ be the function corresponding to $F$ and $\lambda$ defined
by \eqref{GH}. Assume also that $s>1/2$. Then one can
decompose function $G$ in the following way
$$
{G}({x}) = {G}_{-\infty}({x}) + {G_\infty}({x}) + \sum_{  0 \le k \le
\log_2( \lambda^{3/2}/6)} \sum_{m\in \Delta_{k}} {G_{k,m}}({x})
$$
with functions  ${G_{k,m}}$ satisfying the following conditions:
 \begin{equation}\label{l6.5}
\supp {G_{k,m}} \subset [(m-1)2^{-k},(m+1)2^{-k}],
\end{equation}
\begin{equation}\label{gkm1}
 \left|d_{x}^{{l}} {G_{k,m}}({x})\right| \leq
C_{l}{\Theta_{k,m}}2^{k{l}} \quad
\forall{{l}\in \N},
\end{equation}
 and
\begin{equation}\label{gkm2}
\sum_{m\in \Delta_{k}} {\Theta_{k,m}}^2 \leq C2^{-k(2s - 1)}\|F\|^2_{H^s},
\end{equation}
where $C_{l}$ are constants depending only on $s$ and ${l}$ but do
not depend on $k$.

In addition  $\supp {(G_{-\infty})} \subset (-\infty,1/16]$,
$|{G_{-\infty}}|(x) \leq |{G}(x)|$ and
$$\|x^3{G_\infty}\|_{H^2} \leq C\|F\|_{H^s}.$$
\end{lemma}
\begin{proof}
We define ${G_{-\infty}}$ and ${G_\infty}$ multiplying ${G}$ by a smooth
cutoff function, in such a way that ${G_{-\infty}} = {G}$ for $x < 1/32$ and
${G_\infty}(x) = {G}(x)$ for $x > 3$, $\supp({G_\infty}) \subset
[2,\infty)$.  We assume that  $\supp(F) \subset[1/2,1]$ so
 ${G}(x) = - {H}(x)$ for $x > 1$.  Thus estimate for
${G_\infty}$ is a consequence of the last claim of Lemma
\ref{G2-est}. Set ${J} = {G} - {G_{-\infty}} - {G_\infty}$. Then $\supp({J})
\subset [1/32,3]$,
$$
\|J\|_{H^s} \leq C\|F\|_{H^s}
$$
and
\begin{equation}\label{u-der}
\int |d_x^l {J}|^2 \leq C_l((\lambda)^{3/2})^{2(l - s)}\|F\|^2_{H^s}
\end{equation}
for all $l \ge s$.
Note that the last inequality for ${\tilde G}$ follows by
construction, since we cut off frequencies higher than
$(\lambda)^{3/2}/6$ from its Fourier transform.  Changing variable
yields the required estimates for ${J}$. Next let $\eta$ be a smooth
function which is $1$ on $\supp {J}$ and such that $\supp \eta \subset
[0,7/2]$. Recall that  $\psi$ is such a  function that ${\hat \psi}$
is smooth, $\supp({\hat \psi}) \subset [-1, 1]$, ${\hat \psi} = 1$
on $[-1/2, 1/2]$, $0 \leq {\hat \psi} \leq 1$, and ${\hat \psi}$ is
symmetric and that  $\psi_h(x) = h\psi(hx)$

Set
$$
{J}_0(x) = \eta(x)({J}*\psi)(x).
$$
Next we write
$$
{J}_k(x) =  \eta(x)({J}*(\psi_{2^k} - \psi_{2^{k-1}}))(x)
$$
for $ 1 \le k \le  \log_2( \lambda^{3/2}/6) -1$ and
$$
{J}_{k_0} ={J} - \sum_{0 \le k \le  \log_2( \lambda^{3/2}/6)
-1}{J}_k
$$
where $k_0$ is an integer such that $\lambda^{3/2}/12< 2^{k_0} \leq
\lambda^{3/2}/6$. It follows from the definition of ${J}_k$ that
$\supp {J}_k  \subset [0,7/2]$,
$$
{J} = \sum_{  0 \le k \le \log_2( \lambda^{3/2}/6)}{J}_k$$ and
\begin{equation}\label{u-der1}
\int |{d}_x^l {J}_k|^2 \leq C_l2^{2(l-s)k}\|F\|^2_{H^s}.
\end{equation}
Note that to get the last inequality for $k = k_0$ we
use \eqref{u-der}.
Now
let $u$ be a smooth function such that $u = 1$ on $[0,1/2]$,
$\supp(u) \subset [-1/2, 1]$ and $\sum_{m\in \Z} u(x - m) = 1$ for
all $x \in {\mathbb R}$.

Set
$${G_{k,m}}(x) = {J}_k(x)u(2^kx - m).$$
Since $\supp({J}_k) \subset [0, 7/2]$ so ${G_{k,m}} = 0$ for any   $m
\notin \Delta_k$.  Next put
$$
{\Gamma_{k,l,m}} = \sup_x|{d}_x^{l} {G_{k,m}}(x)|$$ and note that
$$
2^{-k}\sum_{m\in \Delta_k} ({\Gamma_{k,l,m}})^2 \le C \sum_{j=0}^l
2^{2(l-j)k}\|{d}_x^{j}{J}_k\|_{2^k, 2}^2
$$
where $\|{d}_x^{j}{J}_k\|_{2^k, 2}^2$ is the norm considered in Lemma \ref{sup-sum}. Now  by Lemma \ref{sup-sum}
$$\|{d}_x^{j}{J}_k\|_{2^k, 2}^2 \le C(\|{d}_x^{j}{J}_k\|_{2}^2+2^{-2k}
\|{d}_x^{j}{J}_k\|_{H^1}^2)\le C'(\|{d}_x^{j}{J}_k\|_{2}^2+2^{-2k}
\|{d}_x^{j+1}{J}_k\|_{2}^2)$$
Hence  \eqref{u-der1} yields
$$
\sum_{m\in \Delta_k} {\Gamma_{k,l,m}}^2 \leq C_{l}2^{-k(2s -2l - 1)}\|F\|^2_{H^s}.
$$
Next set
 $${\Theta_{k,m}} = \sum_{l =0}^\infty
2^{-(k+2)l}C_{l}^{-1/2}{\Gamma_{k,l,m}}.$$
Then by the Cauchy-Schwarz inequality
\begin{eqnarray*}
\sum_{m\in \Delta_k} {\Theta_{k,m}}^2=
\sum_{m\in \Delta_k} \left(     \sum_{l =0}^\infty
2^{-(k+2)l}C_{l}^{-1/2}{\Gamma_{k,l,m}}    \right)^2
 \le C
\sum_{m\in \Delta_k}   \left( \sum_{l =0}^\infty
2^{-2l} \right)      \left( \sum_{l =0}^\infty
  C_{l}^{-1}2^{-2(k+1)l}{\Gamma_{k,l,m}}^2  \right)
\\ \le C
\sum_{m\in \Delta_k}     \sum_{l =0}^\infty
C_{l}^{-1}2^{-2(k+1)l}{\Gamma_{k,l,m}}^2 \le C \sum_{l =0}^\infty 2^{-2(k+1)l}2^{-k(2s -2l - 1)}\|F\|^2_{H^s}.
\\ \leq C2^{-k(2s - 1)}\|F\|^2_{H^s}.
\end{eqnarray*}
Next we observe that
\begin{eqnarray*}
\sup_x|d_x^{{l'}} {G_{k,m}}(x)|
=c_{k,{l'},m}C_{{l'}}^{-1/2}2^{-(k+2){l'}}C_{{l'}}^{1/2}2^{(k+2){l'}}\\
\leq\sum_{l=0}^\infty 2^{-(k+2)l}C_{l}^{-1/2}{\Gamma_{k,l,m}}C_{{l'}}^{1/2}2^{(k+2){l'}}\\
\leq 4C_{{l'}}^{1/2}{\Theta_{k,m}}2^{k{l'}}.
\end{eqnarray*}
This proves the requested estimates for ${G_{k,m}}$.
\end{proof}


The next lemma is based on the finite propagation speed property for the wave equation. 

\begin{lemma}\label{AL}
Assume that  $0< \lambda/4
\le y$,  $\supp F  \subset [1/2, 1]$
and let ${G}$ be the function corresponding to $F$ and $\lambda$ defined by \eqref{GH}.
  Then
$$
{G}({\A}/\lambda)\delta_y= {G}(\LL/\lambda)\delta_y
$$
where $\A $ is the Airy operator and $\LL$ is the anharmonic operator defined by \eqref{op}
\end{lemma}
\begin{proof}
It follow from the finite propagation speed  property for the wave equation
that
$$
\cos t \sqrt {\A} \delta_y = \cos t \sqrt {\LL} \delta_y
$$
for all $|t| \le y$.
Next if $F$ is an even function, then by the Fourier inversion formula,
$$
F(\sqrt {\LL}) =\frac{1}{2\pi}\int_{-\infty}^{\infty}
  \hat{F}(t) \cos(t\sqrt {\LL}) \;dt.
$$
Recall now that ${G}(z)={\tilde G}(\sqrt{z})$ and that the support
of the Fourier transform of ${\tilde G}$ is contained in the
interval $[-\lambda^{3/2}/6, \lambda^{3/2}/6]$, that is  $\supp
({\tilde G})^{\wedge} \subset [-\lambda^{3/2}/6, \lambda^{3/2}/6] $.
Hence if  we assume that $0< \lambda/4 \le y$ then
\begin{eqnarray*}
{G}(\LL/\lambda)\delta_y= {\tilde G}(\sqrt {{\LL}/\lambda})\delta_y
&=&\frac{1}{2\pi}\int_{-\infty}^{\infty}
   ({\tilde G})^{\wedge}(t) \cos(t\sqrt{{\LL}/\lambda}) \delta_y\;dt \\
   =\frac{1}{2\pi}\int_{-\lambda^{3/2}/6}^{\lambda^{3/2}/6}
   ({\tilde G})^{\wedge}(t) \cos(t\sqrt{{\LL}/\lambda}) \delta_y\;dt &=&
   \frac{1}{2\pi}\int_{-\lambda^{3/2}/6}^{c\lambda^{3/2}/6}
   ({\tilde G})^{\wedge}(t) \cos(t\sqrt{{\A}/\lambda}) \delta_y\;dt\\
   &=&{\tilde G}(\sqrt {{\A}/\lambda})\delta_y={G}({\A}/\lambda)\delta_y.
\end{eqnarray*}
This ends the proof of Lemma \ref{AL}.
\end{proof}

Now similarly as in Lemma \ref{2I}  we set  $\widetilde{I}_\lambda =
[\lambda, \infty)$ and
 define   $\chi_{\widetilde{I}_\lambda}$ as the characteristic function of the half-line  $\widetilde{I}_\lambda$.  Then  we denote
by  $\widetilde{I}_\lambda$ also  a projection acting on $L^p(\R)$ spaces defined by
$$
\widetilde{I}_\lambda f(x)= \chi_{\widetilde{I}_\lambda}f(x).
$$
Using this notation we can state Lemma \ref{AL} in the following way
$$
{G}(\LL/\lambda)\widetilde{I}_{\lambda/4}=
{G}({\A}/\lambda)\widetilde{I}_{\lambda/4}.
$$

Note that in virtue of Proposition \ref{I/4} we can assume that $y \in \widetilde{I}_{\lambda/4}$ and this allows us to replace multipliers of the operator $\LL$ by 
spectral multipliers of the Airy operator ${\A}$ by Lemma \ref{AL}. Note next that 
by Lemma \ref{2I} it is enough to consider only a ${L^1(I_{2\lambda})}$ portion of 
the whole $L^1$ norm of the considered kernel. 


\begin{lemma}
Let ${\A}$ be the Airy operator  defined by \eqref{Airy} and 
${G_{-\infty}}$ be the function defined in Lemma~\ref{G1-decomp}.
Then there exists a constant $C>0$, such that
$$
\|{G_{-\infty}}({\A}/ \lambda)\delta_y\|_{L^1(I_{2\lambda})} \leq
C\|F\|_{L^2}
$$
for all $\lambda \geq 4$ and all $y \geq \lambda/4$.
\end{lemma}
\begin{proof}
By Lemma \ref{poz}
$$
{G_{-\infty}}({\A}/\lambda)\delta_y(\cdot) =
\Ai*({G_{-\infty}}(\lambda^{-1}\cdot){\check \Ai}(\cdot - y)).
$$
Recall that by \eqref{e2.5} we have $|{\check \Ai}(x - y)| \leq
C\exp(-(2/3)(|x| + |y|)^{3/2})$ for all $x\leq 0$ and $y \geq \lambda/4 >0$ . Since $|{G_{-\infty}}|
\leq |{G}|$ so by Lemma \ref{G1-neg}
$$
|{G_{-\infty}}(\lambda^{-1}x){\check \Ai}(x - y)| \leq
C\lambda^{3/2}\exp(\lambda^{1/2}|x|^{1/2}/6) \exp(-(2/3)(|x| +
|y|)^{3/2})\|F\|_{L^2}.
$$
By the inequality between arithmetic and geometric means
$\lambda^{1/2}|x|^{1/2} \leq (\lambda + |x|)/2$. Since $|y| \geq
\lambda/4$, we get $\lambda^{1/2}|x|^{1/2} \leq |y| + |x|$. Thus  if
$|y| > 1$ then
$$\lambda^{1/2}|x|^{1/2}/6 \leq (1/3)(|x| + |y|)^{3/2}$$
so
\begin{eqnarray*}
|{G_{-\infty}(\lambda^{-1}x)}{\check \Ai}(x - y)| \leq
C\lambda^{3/2}\exp(-(1/3)(|x| + |y|)^{3/2})\|F\|_{L^2}
\\
\leq
C\lambda^{3/2}\exp(-(1/6)|y|^{3/2})\exp(-(1/6)|x|^{3/2})\|F\|_{L^2}.
\end{eqnarray*}
Hence
$$
\|{G_{-\infty}(\lambda^{-1}\cdot)}{\check \Ai}(\cdot -
y)\|_{L^2((-\infty, 0])} \leq
C'\lambda^{3/2}\exp(-\lambda^{3/2}/24)\|F\|_{L^2}.
$$

Next for $x \in [0, 1/16]$ by the second claim of Lemma \ref{G2-est}
$$
|{G}(x)| = |-H(x)|\leq \sup_{0\leq x\leq 1/4}|H(x^2)| \leq C'\|F\|_{L^2},
$$
so $|{G_{-\infty}}(x)| \leq C'\|F\|_{L^2}$.  Consequently for $x \in [0,
\lambda/16]$
\begin{eqnarray*}
|{G_{-\infty}}(\lambda^{-1}x){\check \Ai}(x - y)| \leq C\exp(-(1/3)(|x -
y|^{3/2})\|F\|_{L^2}
\\
\leq C\exp(-\lambda^{3/2}/24)\|F\|_{L^2}.
\end{eqnarray*}
However by Lemma~\ref{G1-decomp}  $\supp {G_{-\infty}} \subset
(-\infty,1/16]$ so
$$
\|{G_{-\infty}}(\lambda^{-1}\cdot){\check \Ai}(\cdot - y)\|_{L^2([0,\infty)}
\leq C\lambda^{1/2}\exp(-\lambda^{3/2}/24)\|F\|_{L^2}
$$
Combining estimates on $(-\infty, 0]$ and $[0,\infty)$ yields
$$
\|{G_{-\infty}}(\lambda^{-1}\cdot){\check \Ai}(\cdot - y)\|_{L^2} \leq
C\exp(-\lambda)\|F\|_{L^2}.
$$
Hence
$$
\|{G_{-\infty}}({\A}/\lambda)\delta_y\|_{L^2} \leq
C\exp(-\lambda)\|F\|_{L^2}
$$
and
\begin{eqnarray*}
\|{G_{-\infty}}({\A}/\lambda)\delta_y\|_{L^1(I_{2\lambda})} \leq
|I_{2\lambda} |^{1/2}\|{G_{-\infty}}({\A}/\lambda)\delta_y\|_{L^2}
\\
\leq
C\lambda^{1/2}\exp(-\lambda)\|F\|_{L^2} \leq C'\|F\|_{L^2}.
\end{eqnarray*}
\end{proof}

\begin{lemma}
Let $\A$ be the Airy operator defined by \eqref{Airy} and 
${G_{\infty}}$ be the function defined in Lemma~\ref{G1-decomp}. Then there exists a constant $C>0$, such that for all
$\lambda \geq 1$
$$
\|{G_\infty}({\A}/\lambda)\delta_y\|_{L^1(I_{2\lambda})} \leq
C\|F\|_{L^2}.
$$
\end{lemma}
\begin{proof}
By Lemma \ref{poz}
$$
{G_\infty}({\A}/\lambda)\delta_y(\cdot) =
\Ai*({G_\infty}(\lambda^{-1}\cdot){\check \Ai}(\cdot - y)).
$$
We have ${G_\infty}(x)= {H}(x)\eta(x)$ where $\eta(x)$ is a smooth
cutoff function supported on $[2,\infty)$. By Lemma \ref{G2-est}
$$
\sup_{x>2}(|{G_\infty}|(x) )x^4 \leq
C(\lambda^{3/2})^{-2}\|F\|_{L^2}.
$$
Consequently for all  $x \in [2\lambda, \infty)$,
\begin{eqnarray*}
|{G_\infty}(x/\lambda){\check \Ai}(x - y)| &\leq&
C(x/\lambda)^{-4}\lambda^{-3}(1+|x - y|)^{-1/4}\|F\|_{L^2}.
\end{eqnarray*}
and
$$
\|{G_\infty}(\lambda^{-1}\cdot){\check \Ai}(\cdot - y)\|_{L^2} \leq
C\lambda^{-5/2}\|F\|_{L^2}
$$
Thus
$$
\|{G_\infty}({\A}/\lambda)\delta_y\|_{L^2} \leq
C\lambda^{-5/2}\|F\|_{L^2}
$$
and by $\lambda\geq1$
\begin{eqnarray*}
\|{G_\infty}({\A}/\lambda)\delta_y\|_{L^1(I_{2\lambda})} \leq
|I_{2\lambda}|^{1/2}\|{G_\infty}({\A}/\lambda)\delta_y\|_{L^2}
\\
\leq C\lambda^{1/2-5/2}\|F\|_{L^2} \leq C'\|F\|_{L^2}.
\end{eqnarray*}
\end{proof}

For the rest of this section we set $a = 2^{-k}\lambda$. Parameter $a$ shall always
play the same role as in Proposition   \ref{A-est}.

\begin{lemma}\label{gkm-est}
Let ${\A}$ be the Airy operator and  ${G_{k,m}}$ be functions
defined in Lemma \ref{G1-decomp}. For any $l>0$ there exists $C_{l}$
such that if $a \geq \min(1, |y - ma|^{-1/2})$, then
$$
|{G_{k,m}}({\A}/\lambda)\delta_y|(x) \leq C_l{\Theta_{k,m}}
\frac{d^{-1}}{(1+d^{-1}|x - y|)^{l}}
\left(1+\frac{|y-ma|}{1+|x-ma|}\right)^{\frac{1}{4}}
$$
where $d = \max(a^{-1/2}, |y - ma|^{1/2}/a)$
and ${\Theta_{k,m}}$ are constants from Lemma \ref{G1-decomp}. If $a \leq \min(1,
|y - ma|^{-1/2})$, then
\begin{eqnarray*}
|{G_{k,m}}({\A}/\lambda)\delta_y|(x) \leq C_{l}{\Theta_{k,m}}a(1 + |y -
ma|)^{-1/4}(1+ |x - ma|)^{-1/4}(1 + a^{2}|x - ma|)^{-l}.
\end{eqnarray*}
\end{lemma}
\begin{proof}
For any $w \in L^2({\mathbb R})$ we have
$$
w({\A})\delta_y(x) = w({\A} + r)\delta_{y - r}(x - r).
$$
Put $r = ma$ and $w(x) = {G_{k,m}}(\lambda^{-1}(x - ma))$.  Now in
virtue of Lemma~\ref{G1-decomp} we can apply  Proposition
\ref{A-est}.
\end{proof}

Next to investigate $L^p$ properties of the operator ${G}(\A)$  we
are going to decompose it using
 Lemma \ref{G1-decomp}.  For all $0 \le k \le \log_2( \lambda^{3/2}/6)$ we set
 \begin{equation}\label{TK}
T_k = \sum_{m\in \Delta_k }  {G_{k,m}}({\A}/\lambda)
\end{equation}
where ${G_{k,m}}$ are functions defined in Lemma~\ref{G1-decomp}.

\begin{lemma}\label{Tk-1/2}
Let $T_k$ be operator defined by \eqref{TK} corresponding to
functions ${G}$ and $F$ described in Lemma~\ref{G1-decomp}. Assume
further that  $0 < \varepsilon < s - 1/2$. Then there exists
constant~$C$ such that
$$
\|T_k\delta_y\|_{L^1} \leq C2^{-\varepsilon k}\|F\|_{H^s}
$$
for all $k$ such that $2^k \leq \lambda$ and all $y \in I_{2\lambda}$.
\end{lemma}
\begin{proof}
We begin with  decomposing the set $\Delta_k$ in the following way. We
set
$$\Omega_0 = \left\{m\in \Delta_k : {\left|\frac{y}{a} - m\right|} \le 1 \right\}$$
and
$$
\Omega_n =  \left\{m\in \Delta_k : 2^{n-1} < {\left|\frac{y}{a} - m\right|^{1/2}} \leq 2^{n}\right\}
 $$
 for $n>0$.
Then we accordingly decompose operator $T_k$  setting
$$
T_{k}^{\Omega_n} = \sum_{m \in \Omega_n} {G_{k,m}}({\A}/\lambda).
$$
It is enough to prove that
\begin{equation}\label{6.9zero}
\|T_{k}^{\Omega_n}\delta_y\|_{L^1} \leq C2^{-\varepsilon k}\|F\|_{H^s}
\end{equation}
for all $n\in \N$. Indeed, note that if $y \in I_{2\lambda}$ then $ 2^{2(n-1)} < |{2^ky}/{\lambda}| +m
\le 2^{k+1}+2^{k+2} \le 2^{k+3}$.
 Hence $\Omega_n=\emptyset$ unless  $2(n-1) \le k+3$ so given \eqref{6.9zero}
we get
$$
\|T_k\delta_y\|_{L^1} \leq \sum_{n} \|T_{k}^{\Omega_n}\delta_y\|_{L^1}
\leq Ck2^{-\varepsilon k}\|F\|_{H^s}
$$
which yields  the claim for   any $0< \varepsilon' <\varepsilon $.

\medskip

To show \eqref{6.9zero}, firstly note that if $2^k \le \lambda$ then
$a=2^{-k}\lambda \geq1 \ge \min(1, |y - ma|^{-1/2})$. Hence by Lemma
\ref{gkm-est} for  $m \in \Omega_n$ and $n\ge 0$ we have
\begin{eqnarray}\label{C.11Aprl_Gkm}
|{G_{k,m}}({\A}/\lambda)\delta_y| \leq
C_l{\Theta_{k,m}}\frac{d_n^{-1}}{(1 + d_n^{-1}|x -
y|)^{l}}\Big(1+\frac{|y-ma|}{1+|x-ma|}\Big)^{\frac{1}{4}}
\end{eqnarray}
where $d_n = 2^{n+k/2}/\sqrt{\lambda}=2^n/\sqrt{a}$. For
$0<\varepsilon' \leq s - 1/2 - \varepsilon$, set $r =
2^{k\varepsilon'}d_n$
 and
$l> 1 + 1/(2\varepsilon')$. Then by estimate~\eqref{C.11Aprl_Gkm}
and Lemma~\ref{G1-decomp}
\begin{eqnarray}
\nonumber \int_{|x - y| > r}|T_{k}^{\Omega_n}\delta_y| &\leq& \sum_{m
\in \Omega_n} C_l{\Theta_{k,m}}\int_{|x - y| > r}\frac{d_n^{-1}}{(1
+d_n^{-1} |x - y|)^{l}}
\Big(1+\frac{|y-ma|}{1+|x-ma|}\Big)^{\frac{1}{4}}dx
\\
\nonumber &\leq& \sum_{m \in \Omega_n} C_l{\Theta_{k,m}}(r/d_n)^{-l + 1}
\leq C_l\left(2^{k+2}\sum_{0\le m \le 2^{k+2}}
\Theta^2_{k,m}\right)^{1/2}(2^{k\varepsilon'})^{-l + 1}
\\
\label{6.9jeden} &\leq& C'2^{k/2}2^{-k/2}2^{-k(s - 1/2)}\|F\|_{H^s}
\leq C'2^{-k\varepsilon}\|F\|_{H^s}.
\end{eqnarray}

\medskip

Secondly,  for $n \ge 2$ write  $t^{\Omega_n}_{k}(x) = \sum_{m \in
\Omega_n} {G_{k,m}}(\lambda^{-1}x)$. By \eqref{l6.5} and  by
Lemma~\ref{G1-decomp}
$$
\|t^{\Omega_n}_{k}\|^2_{L^2}\leq C \sum_{m \in
\Omega_n}\|{G_{k,m}}(x/\lambda)\|^2_{L^2} \leq Ca\sum_{m \in
\Omega_n} {\Theta_{k,m}}^2.
$$
Recall that $\supp {G_{k,m}} \subset [(m-1)2^{-k},(m+1)2^{-k}]$ so
if $m\in \Omega_n$ and $|2^kx/\lambda-2^ky/\lambda| \le
2^{2(n-1)}-1$ then $t^{\Omega_n}_{k}(x/\lambda){\check \Ai}(x- y)=0$.
Now if $n \ge 2$ and $|2^kx/\lambda-2^ky/\lambda| \ge 2^{2(n-1)}-1$
then $|x-y| \ge a(2^{2(n-1)}-1) \ge a2^{2n}/8$.

  Hence, for  $n\ge 2$,  by Lemma \ref{prop7.1} or by \eqref{e2.6}
\begin{eqnarray}\label{srodek}
\|T_{k}^{\Omega_n}\delta_y\|^2_{L^2} &=&
 \int_{  |x-y| \ge {a2^{2n-3}}} \left| t^{\Omega_n}_{k}(x/\lambda){\check \Ai}(x - y)\right|^2dx
\leq Ca^{-1/2}2^{-n}\|t^{\Omega_n}_{k}\|^2_{L^2} \nonumber\\ &\le&
 C{\sqrt{a}}{2^{-n}}\sum_{m \in \Omega_n} {\Theta_{k,m}}^2.
 \end{eqnarray}
 Estimate \eqref{srodek} holds also for $n=0$ and $n=1$.
Indeed note that
 \begin{eqnarray*}
\int|{G_{k,m}}(x/\lambda){\rm Ai}(y-x)|^2dx
&\leq& \int_{|x-ma|\leq a}|{G_{k,m}}(x/\lambda)|^2(1+|x-y|)^{-1/2}dx\\
&\leq& C\int_{|x-ma|\leq a}\Theta^2_{k,m}(1+|x-y|)^{-1/2}dx \leq
C\Theta^2_{k,m}a^{1/2}.
\end{eqnarray*}
However note that  $\#\Omega_n \le 9$ for $n=0$ and $n=1$ so estimate \eqref{srodek} is
also valid for these $n$.

Thirdly by \eqref{srodek}
\begin{eqnarray}
\nonumber \int_{|x - y| \leq r} |T_{k}^{\Omega_n}\delta_y|dx &\leq&
2r^{1/2}\|T_{k}^{\Omega_n}\delta_y\|_{L^2}\hspace{4cm}
\\
\nonumber &\leq& C\left(2^{k\varepsilon'}d_n\sqrt{a}2^{-n}\sum_{m
\in \Omega_n} {\Theta_{k,m}}^2\right)^{1/2} =
C\left(2^{k\varepsilon'}\sum_{m \in \Omega_n} {\Theta_{k,m}}^2\right)^{1/2}
 \\
\label{6.9dwa} &\leq& C'2^{k\varepsilon'/2}2^{-k(s -
1/2)}\|F\|_{H^s} \leq C'2^{-k\varepsilon}\|F\|_{H^s}.
\end{eqnarray}

Now Lemma~\ref{Tk-1/2} follows from estimates \eqref{6.9jeden} and \eqref{6.9dwa}.
\end{proof}

\begin{lemma}\label{Tk-2/3}
Let $T_k$ be operator defined by \eqref{TK} corresponding to
functions ${G}$ and $F$ described in Lemma~\ref{G1-decomp}. Assume
next that  $0 < \varepsilon < s - 1/2 - 1/6$. Then there exists $C$
such that
$$
\|T_k\delta_y\|_{L^1} \leq C2^{-\varepsilon k}\|F\|_{H^s}
$$
for  all $k$ such that $2^k > \lambda$.
\end{lemma}
\begin{proof}
This time  set
$$\Lambda_0 = \{m \in \Delta_k : |y - ma|^{-1/2} \geq a \},$$
and for $n\ge 1$ put
$$
\Lambda_n = \Omega_n \setminus \Lambda_0= \{m\in \Delta_k : \text{$|y - ma|^{-1/2} < a$ and
 $2^{n-1} < \left|\frac{y}{a} - m\right|^{1/2} \leq 2^n$ }\}$$
where $a = 2^{-k}\lambda$.
Then  for $n=0, 1, 2, \ldots$ we write
$$
T_{k}^{\Lambda_n} = \sum_{m \in \Lambda_n} {G_{k,m}}({\A}/\lambda).
$$

If $n \ge 1$ and $m\in \Lambda_n$ then clearly $a >|y - ma|^{-1/2} \ge \min(1, |y - ma|^{-1/2})$. Hence we can
use the same argument as in the proof of Lemma  \ref{Tk-1/2} to show that
\begin{eqnarray}\label{C.11Aprl_TkGn}
\|T_{k}^{\Lambda_n}\delta_y\|_{L^1} \leq C2^{-\varepsilon
k}\|F\|_{H^s}.
\end{eqnarray}
 Thus  it remains to handle the
operator $T_{k}^{\Lambda_0}$. First   note that $2^k> \lambda $ so
$1>a$. Now by definition, for $m \in \Lambda_0$ we have $|y -
ma|^{-1/2} \geq a$ so $a < \min(1,|y - ma|^{-1/2})$ . Consequently,
by Lemma \ref{gkm-est}
\begin{eqnarray*}
|{G_{k,m}}(\lambda^{-1}{\A})\delta_y| &\leq& C_l{\Theta_{k,m}}a(1 + |y -
ma|)^{-1/4}(1 + |x - ma|)^{-1/4}(1 + a^{2}|x - ma|)^{-l}
\\
&\leq& C'_l{\Theta_{k,m}}a(1 + a^{2}|x - y|)^{-l}.
\end{eqnarray*}
Secondly take $0<\varepsilon' \leq s - 1/2- \varepsilon$, $l \in \N
$ such that $\varepsilon' (l-1)>1$ and set
$r=2^{k\varepsilon'}a^{-2}$. Then  by Lemma \ref{G1-decomp}
\begin{eqnarray}
\int_{|x - y| > r} |T_{k}^{\Lambda_0}\delta_y(x)|dx \leq \sum_{m \in \Lambda_0}
C'_l{\Theta_{k,m}}a\int_{|x - y| > r} (1 + a^{2}|x -
y|)^{-l}dx \nonumber\\
\leq C\frac{(ra^{2})^{-l+1}}{a}\sum_{m \in \Lambda_0} {\Theta_{k,m}}
\leq C\frac{2^{-k(l-1)k\varepsilon'}}{a} \left(2^{k+2}\sum_{m\in \Delta_k}
{\Theta_{k,m}}^2\right)^{1/2} \nonumber\\
 \leq C'2^{k/2}\lambda^{-1}2^{-k(s -
 1/2)}\|F\|_{H^s}\label{C.11Aprl_bigrforTkG0}
 \\ \leq C'2^{-k\varepsilon}\|F\|_{H^s}\nonumber,
\end{eqnarray}
where we used the fact that $0\leq m\leq 2^{k+2}$, the inequality
$2^k \leq \lambda^{3/2}/6$ and the estimates
$$
2^{k/2}\lambda^{-1} \leq 2^{2k/3}\lambda^{-1} \leq
(\lambda^{3/2}/6)^{2/3}\lambda^{-1} = (1/6)^{2/3}.$$

Next, similarly as before  set $t_{k}^{\Lambda_0}(x) = \sum_{m \in \Lambda_0}
{G_{k,m}}(x/\lambda)$. Again by Lemma \ref{G1-decomp}
$$
\|t_{k}^{\Lambda_0}\|^2_{L^2} \leq Ca\sum_{m \in \Delta_k} {\Theta_{k,m}}^2.
$$
Hence
$$
\|T_{k}^{\Lambda_0}\delta_y\|^2_{L^2} = \int_{\R} \left|t_{k}^{\Lambda_0}(x){\Ai}(y-x)\right|^2dx
\leq Ca\sum_{m \in \Delta_k} {\Theta_{k,m}}^2
$$
Now by the Cauchy-Schwarz inequality
\begin{eqnarray}
\int_{|x - y| \leq r} |T_{k}^{\Lambda_0}\delta_y| \leq
2r^{1/2}\|T_{k}^{\Lambda_0}\delta_y\|_{L^2}
&\leq&
C\left(2^{k\varepsilon'}a^{-2}a\sum_{m \in \Delta_k} {\Theta_{k,m}}^2\right)^{1/2}\nonumber\\
\leq C'2^{k\varepsilon'/2}a^{-1/2} 2^{-k(s - 1/2)}\|F\|_{H^s}
&\leq& C'2^{k\varepsilon'/2}\lambda^{-1/2}2^{k/2}2^{-k(s - 1/2)}\|F\|_{H^s}\nonumber\\
&=& C'2^{k\varepsilon'/2}\lambda^{-1/2}2^{k/3}2^{-k(s - 1/2 - 1/6)}\|F\|_{H^s}\nonumber\\
&\leq& C'2^{k\varepsilon'/2}\lambda^{-1/2}(c\lambda^{3/2})^{1/3}
2^{-k(s - 1/2 - 1/6)}\|F\|_{H^s}\nonumber\\ &\leq&
C''2^{-k\varepsilon}\|F\|_{H^s}.\label{C.11Aprl_samllrforTkG0}
\end{eqnarray}
To get the last line we used the inequality $2^k \leq
\lambda^{3/2}/6$ and $\varepsilon<s-1/2-1/6$.

Now Lemma~\ref{Tk-2/3} follows from estimates \eqref{C.11Aprl_TkGn},
\eqref{C.11Aprl_bigrforTkG0} and \eqref{C.11Aprl_samllrforTkG0}.
\end{proof}

\begin{remark}
{\rm Note that in proofs of Lemmas \ref{Tk-1/2} and \ref{Tk-2/3} the assumption
$s> 3/2=1/2+1/6$ was crucial only in the last estimates of the proof of Lemma \ref{Tk-2/3}.
In the rest of the argument it is sufficient to require that $s>1/2$.  }
\end{remark}


\begin{lemma}\label{Lam3/4}
Let $T_k$ be operator defined by \eqref{TK} corresponding to
functions ${G}$ and $F$ described in Lemma~\ref{G1-decomp}. Suppose
also that  $0 < \varepsilon < s - 1/2$ and  $2^k > \lambda$.  Then
$$
\|T_k\|_{L^{4/3}(I_{2\lambda})\rightarrow L^{4/3}(I_{2\lambda})} \leq Ck2^{-k\varepsilon}\|F\|_{H^s}
$$
\end{lemma}
\begin{proof}
Put $t_k(x) = \sum_{m \in \Delta_k} {G_{k,m}}(x/\lambda)$.  We have
$$
T_kf = \int_{\R} t_k(u)\varphi_u\langle f, \varphi_u\rangle
$$
where $\varphi_u(x) = \Ai(x - u)$.

Let $\eta \in C^\infty_c({\mathbb R})$ be  a such function that $0 \leq \eta
\leq 1$, $\supp(\eta) \subset [-8,-2]\cup [2,8]$ and  $\sum_j\eta(2^j x)
= 1$ for all $x \ne 0$.  For $j>0$ put
$$\varphi_{j, u}(x) =
\eta(2^{-j}(x - u))\varphi_u(x) \quad \mbox{ and}  \quad \varphi_{0,u} = \varphi_u -
\sum_{j>0}\varphi_{j,u}.$$ Next  set
$$
T_{k,j}f =  \int t_k(u)\varphi_u\langle f, \varphi_{j,u}\rangle du =
\Ai*(t_k\langle f, \varphi_{j,\cdot}\rangle).
$$
First note that  by \eqref{e2.6}
$$
\|\varphi_{j,u}\|^2_{L^2} \leq C2^{j/2}.
$$
Secondly,  by \eqref{l6.5} and  by Lemma~\ref{G1-decomp}
$$
\int |t_k(u)|^2du \leq a2^{-k(2s - 1)}\|F\|^2_{H^s}
$$
where $a=2^{-k}\lambda$. Hence
\begin{eqnarray*}
\|T_{k,j}f\|^2_{L^2}& =& \int |t_k(u)|^2|\langle f, \varphi_{j,u}\rangle|^2du\\
&\leq& C2^{j/2}\|f\|^2_{L^2}\int |t_k(u)|^2du\\
&\leq & Ca2^{j/2-k(2s - 1)}\|f\|^2_{L^2}\|F\|^2_{H^s}.
\end{eqnarray*}
Thus
\begin{equation}\label{l2one}
\|T_{k,j}\|_{L^2\rightarrow L^2} \le C(2^{j/2}a)^{1/2} 2^{-k(s -
1/2)}\|F\|_{H^s}.
\end{equation}
On the other hand
$$
\|T_k\|_{L^2\rightarrow L^2} \leq \|t_k\|_{L^\infty}\le C \max_{m\in \Delta_k} {\Theta_{k,m}}
\leq C2^{-k\varepsilon}\|F\|_{H^s}
$$
so
\begin{eqnarray}
\|\sum_{2^j > a^{-2}} T_{k,j}\|_{L^2\rightarrow L^2} &\leq&
\|T_k\|_{L^2\rightarrow L^2} + \sum_{2^j \leq
a^{-2}}\|T_{k,j}f\|_{L^2} \nonumber
\\ \label{l2}
&\leq& C(1 + \sum_{2^j \leq
a^{-2}}(2^{j/2}a)^{1/2})2^{-k\varepsilon}\|F\|_{H^s}
\\
&\leq& C'2^{-k\varepsilon}\|F\|_{H^s}. \nonumber
\end{eqnarray}

It follows from the definition of $T_{k,j}$ that
\begin{eqnarray}
T_{k,j}\delta_y =  \int t_k(u)\varphi_u\varphi_u(y)\eta(2^{-j}(y -
u)) du = t^{y}_{k, j}({\A})\delta_y
\end{eqnarray}
where $t^{y}_{k, j}(u) = t_k(u)\eta(2^{-j}(y - u))$. Thirdly note
that $2^{-j} \lambda \le 2^{k}$ so if we set
$$
t^{y}_{k, j}(\lambda u) = \sum_{m \in \Delta_k}\eta(2^{-j}(y -
\lambda u)){G_{k,m}}(u) = \sum_{m \in \Delta_k}G^y_{k,m,j}(u),
$$
then  functions $G^y_{k,m,j}(u) = \eta(2^{-j}(y - \lambda
u)){G_{k,m}}(u)$ satisfy estimates \eqref{gkm1}, \eqref{gkm2} and
inclusion \eqref{l6.5} from Lemma \ref{G1-decomp} uniformly for $j$
and $y$. In addition
$$
T_{k,j}\delta_y =\sum_{m \in
\Delta_k}G^y_{k,m,j}({\A}/\lambda)\delta_y.
$$

 We going to consider two cases: $2^j > a^{-2}$ and $2^j \le  a^{-2}$.
Note that if $2^j > a^{-2}$ and $G^y_{k,m,j}(u)\neq 0$, then
$$
|y-ma| \ge |y-\lambda u|-  |\lambda u -ma|  > 2^{j+1}-a>a^{-2}.
$$
Hence we can  repeat  the argument similar to the proofs of Lemmas
\ref{Tk-1/2} and \ref{Tk-2/3} assuming that $m\notin \Lambda_0$.
It follows that
\begin{equation}\label{kn}
\|T_{k,j}\delta_y\|_{L^1} \leq C2^{-k\varepsilon}\|F\|_{H^s}.
\end{equation}
(Note that additional $1/6$ was necessary only to consider the case
$m\in \Lambda_0$.)
Next we notice  that $2^j \le |y-u|$ unless $\eta(2^{-j}(y - u))=0$. Then $|u|
\le 4 \lambda $ unless $t_k(u)=0$. Hence $2^n \le |y-u| \le |y| +|u|
\le 2 \lambda +4\lambda =6 \lambda \le 62^k\le 2^{k+3}$ unless
$t^{y}_{k, j}(u)=0$. Therefore we can assume that $0\le j \le 4k$. Thus
there are at most $4k$ nonzero terms in the following  sum  so by \eqref{kn}
$$
\|\sum_{2^j>a^{-2}}T_{k,j}\|_{L^1\rightarrow L^1} \leq
Ck2^{-k\varepsilon}\|F\|_{H^s}.
$$
Interpolating with estimate \eqref{l2} yields the required estimates
\begin{eqnarray}\label{C.11Aprl_bigjforTkj}
\|\sum_{2^j>a^{-2}}T_{k,j}\|_{L^{4/3}\rightarrow L^{4/3}} \leq
Ck2^{-k\varepsilon}\|F\|_{H^s}.
\end{eqnarray}

It remains to handle the case $2^j < a^{-2}$.
Again we decompose  the operator $T_{k,j}$, this time into two parts
$$
T_{k,j}^{\Lambda_0} = \sum_{m \in \Lambda_0}
G^y_{k,m,j}(\lambda^{-1}{\A}) \quad \mbox{and} \quad
T_{k,j}^{\Lambda_0^c}=T_{k,j}-T_{k,j}^{\Lambda_0} = \sum_{m \notin
\Lambda_0} G^y_{k,m,j}(\lambda^{-1}{\A})
$$
Recall that functions $G^y_{k,m,j}$ satisfies the same condition as $G_{k,m}$ so
if $m \in \Lambda_0$ then by
 Lemma \ref{gkm-est}
$$
|G^y_{k,m,j}(\lambda^{-1}{\A})\delta_y|(x) \leq C_l{\Theta_{k,m}}a(1 +
a^2|x - y|)^{-l}.
$$
With $r=2^{k\varepsilon'}a^{-2}$ where $0<\varepsilon' \leq s - 1/2
- \varepsilon$ and $l> 1 + 1/\varepsilon'$, like in the proof of
 Lemma~\ref{Tk-2/3} we have
\begin{eqnarray}\label{C.11Aprl_bigrforTkjG0}
\int_{|x - y| > r} |T_{k,j}^{\Lambda_0}\delta_y| \leq
C'2^{-k\varepsilon}\|F\|_{H^s}.
\end{eqnarray}
We have
$$
|\varphi_{j,u}|(x) \leq C2^{-j/4}.
$$
Consequently, if $2^j \leq a^{-2}$,
\begin{eqnarray*}
\|T_{k,j}^{\Lambda_0}\delta_y\|_{L^2}^2 &=&
 \int \left|\sum_{m \in \Lambda_0} G_{k,m}(u/\lambda)\right|^2\left|\varphi_{j,u}(y)\right|^2du
\\
&\leq& C2^{-j/2}\int \left|\sum_{m \in \Lambda_0}
G_{k,m}(u/\lambda)\right|^2du \leq C'2^{-j/2}a 2^{-k(2s -
1)}\|F\|^2_{H^s}
\end{eqnarray*}
and
\begin{eqnarray*}
\int_{|x - y| \leq r} |T^{\Lambda_0}_{k,j}\delta_y| &\leq&
(2r)^{1/2}\|T^{\Lambda_0}_{k,j}\delta_y\|_{L^2} \leq
C2^{k\varepsilon'/2}a^{-1}2^{-j/4}a^{1/2} 2^{-k(s - 1/2)}\|F\|_{H^s}
\\
&\leq& C2^{-j/4}a^{-1/2}2^{-k\varepsilon}\|F\|_{H^s}
\end{eqnarray*}
which, combing with estimate \eqref{C.11Aprl_bigrforTkjG0}, implies
that
\begin{equation}\label{zero}
\|T_{k,j}^{\Lambda_0}\|_{L^1\rightarrow L^1} \leq
C2^{-j/4}a^{-1/2}2^{-k\varepsilon}\|F\|_{H^s}.
\end{equation}
Remembering  that $m \notin \Lambda_0$ and repeating the argument of Lemmas \ref{Tk-1/2} and \ref{Tk-2/3} yield
\begin{equation}\label{nonzero}
\|T_{k,j}^{\Lambda_0^c}\|_{L^1\rightarrow L^1} \leq
C2^{-k\varepsilon}\|F\|_{H^s}.
\end{equation}
However $2^j \le a^{-2}$ so combining estimates \eqref{zero} and \eqref{nonzero}
shows that
$$
\|T_{k,j}\|_{L^1\rightarrow L^1} \leq
C2^{-j/4}a^{-1/2}2^{-k\varepsilon}\|F\|_{H^s}.
$$
Now interpolating with estimate \eqref{l2one} gives
\begin{eqnarray}\label{C.11Aprl_smalljforTkj}
\|T_{k,j}\|_{L^{4/3} \rightarrow L^{4/3}} \leq
C2^{-k\varepsilon}\|F\|_{H^s}.
\end{eqnarray}
Now Lemma~\ref{Lam3/4} follows from estimates
\eqref{C.11Aprl_bigjforTkj} and \eqref{C.11Aprl_smalljforTkj}.
\end{proof}
\bigskip

The proof of Lemma~\ref{Lam3/4} concludes also the proof of Theorem~\ref{main}.

\section{Necessary conditions for Bochner-Riesz summability and proof of Theorem~\ref{Riesz}}\label{sec3}
\setcounter{equation}{0}
This section is devoted to the discussion of a necessary condition of the
boundedness of Bochner-Riesz means of anharmonic oscillator  $\LL=-\frac{d^2}{dx^2}+|x|$.

\medskip

\begin{theorem}\label{verlp}
 Suppose that  $\LL$ is an anharmonic oscillator  defined by formula \eqref{op} and that
 the  Bochner-Riesz means of order $\alpha$ are uniformly
bounded on $L^p$, that is
$$
\sup_{R>0}\|\sigma_R^{\alpha}(\LL)\|_{p \to p} \le C < \infty.
$$
Then it necessarily follows that
$$
\alpha \ge  \max\left\{0,\frac{2}{3}\left|\frac{1}{2}-\frac{1}{p}\right|-\frac{1}{6}\right\}.
$$
In addition for $p=4$ and $p=4/3$ the necessary condition is $\alpha >0$.
\end{theorem}

\begin{proof}
We start the proof by introducing the distributions $\chi_-^a$, defined by
$$
\chi_-^a=\frac{x_-^a}{\Gamma(a+1)},
$$
where $\Gamma$ is the Gamma function and
$$
x_-^a=|x|^a \quad \mbox{if} \quad x \le 0 \quad \quad \mbox{and}\quad  x_-^a=0 \quad  \mbox{if} \quad x > 0.
$$
Then $x_-^a$ are clearly distributions for $\Re a > -1$ and we have
for $\Re a > 0$,
\begin{equation*}
\frac{d}{dx} x_-^a = -a x_-^{a-1} \implies \frac{d}{dx} \chi_-^a=-\chi_-^{a-1}
\end{equation*}
which we use to extend the family of functions $ \chi_-^a$ to a family of distributions on $\rr$ defined for all $a\in \cc$, see \cite{Ho1} for details.  Since $1-\chi_-^0(x) $ is the Heaviside function, it follows that
\begin{equation}
\chi_-^{-k} =(-1)^k\delta_0^{(k-1)}, \quad  k=1,2,\ldots ,
\label{Heaviside}
\end{equation}
where $\delta_0$ is  the $\delta$-Dirac measure.
Motivated by the above equality    we define distribution $\delta_-^{\nu}$ for  all real exponents $\nu \in \R$ by
the formula
$$
\delta_-^{\nu}=\chi_-^{-\nu-1}.
$$
A straightforward computation shows that for all $w,z\in \cc$
\begin{equation*}\label{aa2}
\chi_-^w * \chi_-^z=\chi_-^{w+z+1}
\end{equation*}
where $\chi_-^w* \chi_-^z$ is the convolution of the distributions $\chi_-^w$ and $\chi_-^z$,
see \cite[(3.4.10)]{Ho1}. It follows from the above relation that
$$
\delta_-^{\nu}*\delta_-^{\mu}=\delta_-^{\nu+\mu}.
$$
Now if $\supp F \subset [0,\infty)$ we define the Weyl fractional
derivative of $F$ of order $\nu$ by the formula
$$
F^{(\nu)}=F*\delta_-^{\nu}
$$
and we note that
$$
F^{(\nu)}*\delta_-^{-\nu}=F*\delta_-^{\nu}*\delta_-^{-\nu}=F,
$$
see \cite[Page 308]{GP} or  \cite[(6.5)]{DOS}.
Thus
$$
F(\LL)=\frac{1}{\Gamma(\nu)}\int_0^\infty F^{(\nu)}(s)
(s-{\LL})_+^{\nu-1}ds= \frac{1}{\Gamma(\nu)}\int_0^\infty
F^{(\nu)}(s)s^{\nu-1}\sigma_s^{\nu-1}(\LL)ds,\ \ \forall{\nu>0}
$$
for all $F$ supported in the positive half-line. 
 Hence  if $\supp F \subset [0,\infty)$ then
\begin{eqnarray}
\|F(\LL)\|_{L^p\to L^p}\leq  C\sup_{R>0}\|\sigma_R^{\nu}(\LL)\|_{L^p\to
L^p} \int_0^\infty \left|F^{(\nu+1)}(s)\right|s^{\nu}ds.
 \label{e4.1}
\end{eqnarray}

\medskip
Consider now function $\eta \in C_c^\infty(\R)$  such that $\eta(0)=1$ and  $\supp
\eta \subset [-\frac{\pi}{2},\frac{\pi}{2}]$ and let $\{\lambda_n\}$
be the  set of  eigenvalues of the operator $\LL$. We define sequence of functions
$F_n$ by the formula
$$
F_n(\lambda)=\eta(\sqrt{\lambda_{n+1}}(\lambda-\lambda_n)).
$$
It follows from \eqref{e2.4} that  $F_n(\lambda_m) =1$, if $n=m$ and
$F_n(\lambda_m) =0$ otherwise.
Hence
$$
F_n(\LL)f=\sum_{m=1}^{\infty}F_n(\lambda_m)<{\phi}_n,f>{\phi}_n=<{\phi}_n,f>{\phi}_n.
$$
Thus by Lemma~\ref{le2.3} and estimates \eqref{eq2.31}  for all $p<2$
$$
\|F_n(\LL)\|_{p\to p} =\|{\phi}_n\|_{p}\|{\phi}_n\|_{{p'}} \ge
c\lambda_n^{-1/4}\lambda_n^{1/p-1/2} $$ where $p'$ is conjugate
exponent of $p$, that is $1/p+1/p'=1$. In addition for $p=4/3$.
$$
\|F_n(\LL)\|_{4/3\to 4/3} =\|{\phi}_n\|_{4/3}\|{\phi}_n\|_{{4}} \ge
c( \ln \lambda_n)^{1/4}.
$$

Next note that $\delta_-^{\nu}$
is a homogenous distribution, see \cite[Definition 3.2.2]{Ho1} so
$F_n^{(\nu+1)}(\lambda)= \lambda_{n+1}^{{(\nu+1)}/{2}}\eta^{(\nu+1)}(\sqrt{\lambda_{n+1}}(\lambda-\lambda_n)) $. Hence setting $a=\lambda_n-\pi\lambda_{n+1}^{-1/2}/2$
and  $b=\lambda_n+\pi\lambda_{n+1}^{-1/2}/2$ one gets
\begin{eqnarray} \nonumber
\int_0^\infty \left|F_n^{(\nu+1)}(\lambda)\right|\lambda^{\nu}d\lambda=\lambda_{n+1}^{{(\nu+1)}/{2}}
\int_{a}^{b} \left|\eta^{(\nu+1)}(\sqrt{\lambda_{n+1}}(\lambda-\lambda_n))\right|\lambda^{\nu}d\lambda
\\ \le C\lambda_{n+1}^{{(\nu+1)}/{2}} \lambda_{n+1}^{\nu}(b-a)
\le C \lambda_{n}^{3\nu/2}. \label{ee1}
\end{eqnarray}

Now suppose that $\sup_{R>0}\|\sigma_R^{\alpha}(\LL)\|_{p\to p} < \infty $.
Substituting $\alpha=\nu$ in \eqref{e4.1} and using estimate \eqref{ee1} show that if $1\le p <2$  then
$$c\lambda_n^{-1/4}\lambda_n^{1/p-1/2}\leq C \lambda_{n}^{3\alpha/2}.$$
The above estimates  can hold  for large $n$ only if $\alpha \ge
-\frac{1}{2}+\frac{2}{3p}$ or $p\le \frac{4}{6\alpha +3}$. A similar argument shows
that for  $p=4/3$ the necessary condition is $\alpha >0$. We extend this necessary condition to all $1\le p \le \infty$ by duality. This ends the proof of Theorem~\ref{verlp}.
\end{proof}

\begin{remark}
{\rm The method used here can be used to get the necessary condition
of the boundedness of Bochner-Riesz means for harmonic oscillator
which was proved in \cite{T} in another way. More precisely if
${\mathcal H}=-\frac{d^2}{dx^2}+x^2$ and
$$
\sup_{R>0}\|\sigma_R^{\alpha}({\mathcal H})\|_{p \to p} \le C < \infty.
$$
Then it necessarily follows that
$
\alpha \ge  \max\left\{0,\frac{2}{3}\left|\frac{1}{2}-\frac{1}{p}\right|-\frac{1}{6}\right\}.
$}
\end{remark}
\begin{proof}
This time consider function  $\eta \in C_c^\infty(\R)$  such that $\eta(0)=1$ and  $\supp
\eta \subset [-1,1]$ and set $F_n(\lambda)=\eta(\lambda-2n-1)$.
Similarly as before
$$F_n({\mathcal H})f=\sum_{m=1}^{\infty}F_n(\lambda_m)<h_n,f>h_n=<h_n,f>h_n$$ where this time $h_n$ is $n$-th Hermit function. Clearly
$
\int_0^\infty \left|F_n^{(\nu+1)}(\lambda)\right|\lambda^{\nu}d\lambda \sim n^{\nu}$.
On the other hand it follows from the standard asymptotic for the Hermit functions
that $\|h_n\|_{p}\|h_n\|_{{p'}} \ge Cn^{-\frac{1}{2}+\frac{2}{3p}}$, see for example
\cite[Lemma 2.1]{T}.  Next if  $\sup_{R>0}\|\sigma_R^{\alpha}({\mathcal H})\|_{p\to p} < \infty $ then
$$
n^{-\frac{1}{2}+\frac{2}{3p}} \le C n^\alpha.
$$
This yields the required necessary condition.
\end{proof}

For a sake of completeness we end this section with a discussion of
the proof of Theorem~\ref{Riesz}, which at this point is an immediate consequence of Theorem~\ref{main} and \ref{verlp}.

\begin{proof}[Proof of Theorem  \ref{Riesz}]
We proved in Theorem~\ref{verlp}
that if $\LL=-\frac{d^2}{dx^2}+|x|$ and
$$
\sup_{R>0}\|\sigma_R^{\alpha}({\LL})\|_{p \to p}= \sup_{t>0}\|\sigma_1^{\alpha}(t{\LL})\|_{p \to p}\le C < \infty.
$$
Then it necessarily follows that
$
\alpha \ge  \max\left\{0,\frac{2}{3}\left|\frac{1}{2}-\frac{1}{p}\right|-\frac{1}{6}\right\}
$. Hence it remains to prove that if $
\alpha >  \max\left\{0,\frac{2}{3}\left|\frac{1}{2}-\frac{1}{p}\right|-\frac{1}{6}\right\}
$ then indeed the above estimate for the Riesz means holds.
To show it take a function $\psi\in C_c^\infty(-3/4,3/4)$ such that $\psi=1$ on
$[-1/2,1/2]$ so that
$$
\sigma_1^{\alpha}(\lambda)=(1-{\lambda^2})_+^{\alpha}=(1-{\lambda^2})_+^{\alpha}\psi(\lambda)+(1-{\lambda^2})_+^{\alpha}(1-\psi(\lambda))=F^\alpha_1(\lambda)+F^\alpha_2(\lambda).
$$
where $F^\alpha_1(\lambda)=\sigma_1^{\alpha}(\lambda)\psi(\lambda)$
and $F^\alpha_2(\lambda)=\sigma_1^{\alpha}(\lambda)(1-\psi(\lambda))$.
Now it is enough   to show that if $
\alpha >  \max\left\{0,\frac{2}{3}\left|\frac{1}{2}-\frac{1}{p}\right|-\frac{1}{6}\right\}
$ then
$$
\sup_{t>0}\|F^\alpha_1(t{\LL})\|_{p \to p} \le C < \infty \quad \mbox{and} \quad \sup_{t>0}\|F^\alpha_2(t{\LL})\|_{p \to p} \le C < \infty.
$$
Note that $\supp F^\alpha_2 \subset [1/2,1]$ and if $\alpha +1/2 >s$ then $\sigma_1^{\alpha} \in H^s$ and
$F^\alpha_2 \in H^s$. Now the required estimate for $\|F^\alpha_2(t{\LL})\|_{p \to p} $
follows directly from Theorem~\ref{main}. On the other hand it is not difficult
to note that $F^\alpha_1\in C_c^\infty(-3/4,3/4)$ and required estimates for $\|F^\alpha_1(t{\LL})\|_{p \to p} $ follows from Proposition \ref{th3.1}.
\end{proof}

In fact estimates for $F^\alpha_1$ do not required the sharp result and  follows from standard spectral multipliers theorems. Indeed it follows from the FeynmanÐKac formula
the corresponding heat kernel satisfies Gaussian upper bounds.  Now
the required estimate for $\|F^\alpha_1(t{\LL})\|_{p \to p} $ follows for example from  \cite[Theorem 3.1 and Remark 1, page 451]{DOS} or \cite[Theorem 3.1]{COSY}.

\section{Concluding remarks}\label{sec7}

Next we will  discuss a singular integral  version corresponding to Theorem~\ref{main}.
That is we extend compactly (dyadicaly) supported spectral multipliers
to singular integral version similar as in  Fourier multipliers of Mikhlin-H\"ormander type.
The result can be stated in the following way.

\begin{theorem}\label{singularmain}
Suppose that   $\LL$ is an anharmonic oscillator  defined by \eqref{op}.
 Assume next that  $ 1< p < \infty $,
 $s > \max\{\frac{1}{2},\frac{2}{3}|\frac{1}{2}-\frac{1}{p}|+\frac{1}{3}\}$
 and the bounded Borel function $F$ satisfies
 $$
 \sup_{t>0}\|\eta\delta_tF\|_{H^{s}}<\infty
 $$
 where $\eta\in C_c^\infty(1/4,4)$ is a fixed non-trivial auxiliary function.

 Then the operators $F(\LL)$ are bounded on space $L^p(\R)$  and the following estimate
$$
\|F(\LL)\|_{p\to p} \le C \sup_{t>0}\|\eta\delta_tF\|_{H^{s}} < \infty.
$$
holds for the multiplier $F(\LL)$.
\end{theorem}
\begin{proof}
The result and the above estimates follow directly from \cite[Theorem 3.3]{SYY}.
\end{proof}

The following statement is an obvious consequence  of Theorem \ref{main}.
We state this results here  to explain better the relation between Theorem \ref{main}
and the Bochner-Riesz summability result that is Theorem~\ref{Riesz}.

\begin{coro}\label{koniec}
Suppose that   $\LL$ is an anharmonic oscillator  defined by \eqref{op} and that
 $\supp F\subset [1/2,1]$. Assume next that  $ 1\le p \le \infty $,
 $s > \max\{1,\frac{2}{3}|\frac{1}{2}-\frac{1}{p}|+\frac{1}{3}+\frac{1}{2}\}$
 and that $F \in W^1_{s}$.

 Then the operators $F(t\LL)$ are uniformly bounded on space $L^p(\R)$ and
$$
\sup_{t>0} \|F(t\LL)\|_{p\to p} \le C \|F\|_{W^1_{s}}.
$$
\end{coro}
\begin{proof}
The results is straightforward consequence of Theorem~\ref{main} and the fact
that $W^1_{s_1}\subset W^2_{s_2}=H^{s_2}$ for all $s_1,s_2$ such that
$s_1> s_2+\frac{1}{2}$.
\end{proof}

\begin{remark}
Note that one can use estimate \eqref{e4.1} to show that Corollary \ref{koniec} follows
also from Theorem~\ref{Riesz}. Then it is also not difficult to notice that using an argument
similar as in proof of Theorem~\ref{Riesz} above that this theorem also follows from
the above corollary. That is  Corollary \ref{koniec} and Theorem~\ref{Riesz}
are equivalent statements of the same result. One can also formulate singular integral
version of  Corollary \ref{koniec} similarly as in Theorem~\ref{singularmain} above.
\end{remark}

One can use the same argument as in paragraph above to formulate Proposition \ref{AWT} in
an equivalent way using the same terms as in  Corollary \ref{koniec} with operator $\LL$ replaced by $\HH$.
This leads to a question if we could replaced the operator $\LL$ by $\HH$ also in
the formulation of Theorem~\ref{main}. The answer to this question is likely to be positive but
as the format of Theorem~\ref{main} is essentially stronger than the statement of Corollary \ref{koniec}. Therefore  the proof of such statement requires new techniques and the argument used in \cite{T} can not be adapted to yields the version of Theorem~\ref{main} for $\HH$ without
significant changes. Most likely such proof would require completely new approach. We leave this question open for future studies.

On the other hand the result obtained in \cite{AW} guarantees convergence of
the Bochner-Riesz mean of order $0$ that is simply the convergence of eigenfunction expansion
in $L^p(\rr)$ spaces for all $ 4/3<p <4$ whereas Theorem \ref{Riesz} requires the strictly positive order to assure convergence for operator $\LL$ in this range. It is likely that the main result of  \cite{AW} holds also for $\LL$ but this again would require a completely new approach and we again leave this question open for future studies.

If one consider spaces $L^p$ for $4/3<p< 4$ then Theorem~\ref{main} gives essentially stronger
estimates than Proposition~\ref{th3.1}. Also when applied to $L^1$ Theorem~\ref{main} is significantly deeper and more interesting than Proposition~\ref{th3.1}. Note however, that none of the
spaces $W^2_{2/3}$ and $W^{4}_{1/2}$ contains the other so formally speaking these two results are of independent interested and none follows from the other. One could ask whether to assure
boundedness of the multiplier $F(\LL)$ on $L^1$ it is enough  to assume that $F$ is in the
space  $W^{3}_{s}$ for some $s>1/2$. A positive answer to this problem would imply on  the level of $L^1$ both estimates from Theorem~\ref{main} and Proposition~\ref{th3.1}.  It is likely though that the answer to this question is negative
but we will not study this issue here. We point out however that the consideration of imaginary powers
$\LL^{is}$ shows the $1/2$ is the minimal possible order of differentiability for spectral multipliers in the dimension one, see \cite[Theorem 1]{SW}. Then if estimates would hold with
the norm of $W^{p}_{1/2}$ norm of $F$ then necessarily $p>3$. Otherwise such estimates would imply convergence of Bochner-Riesz means of order smaller then $1/6$ and this would contradict
Theorem~\ref{verlp}. As we mentioned above we expect that even the norm  $W^{3+\epsilon }_{1/2+\epsilon }$ of $F$ for some very  small positive $\epsilon$ is still not enough to ensure
$L^1$ boundedness  of the multiplier $F(t\LL)$.

\end{document}